\documentclass[reqno, 11pt]{amsart}
\usepackage{}
\usepackage{bbm}
\usepackage{url}
\usepackage{stmaryrd}
\usepackage{mathrsfs}
\usepackage{cases}
\usepackage{amsfonts}
\usepackage{amssymb}
\usepackage{amsmath}
\usepackage{enumerate}
\usepackage{tikz}
\usepackage{extarrows}
\usepackage{mathtools}

\usepackage{citeref}

\allowdisplaybreaks[4]
\newtheorem{theorem}{Theorem}[section]
\newtheorem{lemma}[theorem]{Lemma}
\newtheorem{proposition}[theorem]{Proposition}
\newtheorem{corollary}[theorem]{Corollary}
\theoremstyle{definition}
\newtheorem{definition}[theorem]{Definition}

\newtheorem{remark}[theorem]{Remark}
\numberwithin{equation}{section}

\let\g=\gamma

\let\z=\zeta
\let\la=\lambda

\let\ep=\epsilon
\let\va=\varphi

\let\fy=\infty
\def\bbR{\mathbb{R}}

\def\bbC{\mathbb{C}}

\def\calA{\mathcal {A}}
\def\calB{\mathcal {B}}

\def\calH{\mathcal {H}}

\def\calL{\mathcal {L}}
\def\calM{\mathcal {M}}

\def\calS{\mathcal {S}}

\def\scrA{\mathscr{A}}
\def\scrF{\mathscr{F}}

\def\scrV{\mathscr{V}}
\def\scrM{\mathscr{M}}

\def\calN{\mathcal {N}}
\def\calHS{\mathcal {H}\mathcal {S}}
\def\mfV{\mathfrak{V}}

\def\rr{{\mathbb R}}
\def\rd{{{\rr}^d}}
\def\rdd{{{\rr}^{2d}}}
\def\rddd{{{\rr}^{4d}}}
\def\zd{{{\mathbb{Z}}^d}}
\def\zdd{{{\mathbb{Z}}^{2d}}}

\newcommand{\be}{\begin{equation*}}
\newcommand{\ee}{\end{equation*}}
\newcommand{\ben}{\begin{equation}}
\newcommand{\een}{\end{equation}}
\newcommand{\bn}{\begin{enumerate}}
\newcommand{\en}{\end{enumerate}}
\newcommand{\bs}{\backslash}
\newcommand{\lan}{\langle}
\newcommand{\ran}{\rangle}

\def\rr{{\mathbb R}}
\def\cc{{\mathbb C}}
\def\rd{{{\bbR}^d}} \def\rtd{\bbR^{2d}}

\def\hs{{\mathcal{HS}}}
\def\sp{{\mathcal{S}}_p}

\def\mv{{M_v^1}}
\def\lv{L_v^1}
\def\mvd{{M_v^1(\bbR^d)}}

\def\mvs{{(M_v^1)^*}}
\def\mvds{{(M_v^1(\bbR^d))^*}}

\def\mpqm{{M^{p,q}_m}}   
\def\lpqm{L^{p,q}_m}
\def\mpqmd{{M^{p,q}_m(\bbR^d)}}
\def\lpqmtd{L^{p,q}_m(\bbR^{2d})}
\def\bpqm{\calB_{p,q}^m}
\def\bv{\calB_{1}^v}
\def\bvv{\calB_{\fy}^{v^{-1}}}

\begin{document}
\title[A note on the operator window of modulation spaces]
{A note on the operator window of modulation spaces}
\author{WEICHAO GUO}
\address{School of Science, Jimei University, Xiamen, 361021, P.R.China}
\email{weichaoguomath@gmail.com}
\author{GUOPING ZHAO}
\address{School of Mathematics and Statistics, Xiamen University of Technology,
Xiamen, 361024, P.R.China} \email{guopingzhaomath@gmail.com}

\begin{abstract}
Inspired by a recent article \cite[JFAA, 28(2):1-34, (2022)]{Skrettingland2022JoFAaA}, 
this paper is devoted to the study of suitable window class in the framework of bounded linear operators on $L^2(\rd)$.
We establish a natural and complete characterization for the window class such that the corresponding STFT leads to equivalent norms of modulation spaces. The positive bounded linear operators are also characterized
in Cohen's class distributions such that the corresponding quantities form equivalent norms of modulation spaces.
As a generalization, we introduce a family of operator classes corresponding to the operator-valued modulation spaces.
Some applications of our main theorems to the localization operators are also concerned.

\end{abstract}

\maketitle


\section{INTRODUCTION}
Modulation space was first introduced by H. Feichtinger \cite{Feichtinger1983TRUoV} in 1983.
Now, it has been proven to be an important function spaces in the field of time-frequency analysis \cite{GrochenigBook2013}.
Moreover, modulation space has been associated with many topics of mathematics such as partial differential equation 
\cite{Benyi2007JFA, WangHudzik2007JDE, M.RuzhanskyM.Sugimoto2012EEoHaST} and classical harmonic analysis 
\cite{GroechenigHeil1999IEOT,GuoFanWuZhao2018SM}.

The purpose of modulation space is to describe the content of the functions or distributions on the time-frequency plane.
To achieve this goal, the short time Fourier transform (STFT) is used to extract the local information of functions or distributions.
More precisely, STFT can be firstly defined on $L^2(\rd)$ by
\be
V_{\va}f(z)=\langle f, \pi(z)\va\rangle_{L^2}\ \ \ (z\in \rdd),
\ee
where the window $\va$ is a function with some good localized properties on the time-frequency plane,
and $\pi(z)$ denotes the time-frequency shift for $z=(x,\xi)$ defined by 
\be
\pi(z)\va(t)=M_{\xi}T_x  \va(t)=e^{2\pi it\cdot\xi}\va(t-x).
\ee
With a suitable window $\va$, the STFT can be well defined for $f$ belonging to the space of tempered distributions $\calS'(\rd)$ or the dual space of $M^1_v(\rd)$ denoted by $(M^1_v(\rd))^*$. 

Let $g_0$ be the normalized Gaussian, i.e.,
\be
g_0(t)=2^{d/4}e^{-\pi|t|^2}.
\ee
We point out that $g_0$ will always work as a suitable window whether $f$ belongs to $L^2(\rd)$, $\calS'(\rd)$ or $\mvds$.
The modulation space can be defined by
\be
M^{p,q}_m(\rd)=\{f\in \mvds: V_{g_0}f\in L^{p,q}_m(\rdd) \},
\ee
endowed with the obvious (quasi-)norm, where $L^{p,q}_m(\rdd)$ are weighted mixed-norm Lebesgue spaces with the weight $m\in \calM_v$.
Here $\calM_v$ denotes the class of all $v$-moderate weight functions, where $v$ is a submultiplicative weight.
See the precise definitions of weight functions in Subsection 2.2.    
Sometimes, we write $M^p_{m}=M^{p,p}_m$ for short.

In the above definition of modulation space, the Gaussian $g_0$ serves as the window. A natural problem is:  can the window $g_0$ be replaced by another suitable function in the definition of modulation space? More precisely, can we give a characterization for all $\va$ satisfying the following
equivalent relation?
\ben\label{Eqi-0}
\|V_{\va}f\|_{\lpqm}\sim_{\va,m,v} \|V_{g_0}f\|_{\lpqm} \ \ \ \text{for all}\  f\in \mpqmd,\ 1\leq p,q \leq \fy, m\in \calM_v(\rdd).
\een 
This problem is not difficult to answer in some sense.
By the fact that 
\be
\bigcup_{\substack { 1\leq p,q\leq \fy,\ m\in \calM_v}}\mpqm=M^{\fy}_{v^{-1}} = \mvs,
\ee
we see that in the definition of STFT, the possible largest class of windows fitting for all $\mpqmd$ mentioned above, is the modulation space $M^1_v(\rd)$, which serves as the test function of $M^{\fy}_{v^{-1}}(\rd)$.
On the other hand,  for all $\va \in M^1_v(\rd)\bs \{0\}$, $\|V_{\va}f\|_{\lpqm(\rdd)}$ defines an equivalent norm on $\mpqmd$ (see \cite[Proposition 11.4.2]{GrochenigBook2013}).
Thus, we conclude that \eqref{Eqi-0} holds if and only if $\va \in M^1_v(\rd)\bs \{0\}$.

Let $\pi(z)^*$ be the Hilbert adjoint of $\pi(z)$ defined by
\be
\pi(z)^*=e^{-2\pi ix\cdot\xi}\pi(-z),\ \ \ z=(x,\xi).
\ee
Note that both $\pi(z)$ and $\pi(z)^*$ are bounded on $\mv$ and can be extended by duality to boundeded operators on $\mvs$.
Write STFT by
\ben\label{lf}
V_{\va}f(z)=\lan f, \pi(z)\va\ran _{\mvs,\mv}=
\langle \pi(z)^*f, \va\rangle_{\mvs,\mv}=:L_{\va}(\pi(z)^*f).
\een
Here, $L_{\va}$ means the obvious bounded linear functional on $\mvs$.
Denote by  $\calHS$ the collection of all Hilbert-Schmidt operators on $L^2(\rd)$, $\calN(L^2,\mv)$ the set of all nuclear operators bounded from $L^2(\rd)$ into $\mvd$. See Subsection 2.3 
for the precise definition of nuclear operators.
In \cite{Skrettingland2022JoFAaA}, the author considers a class of linear operators defined by 
\be
\calN^*=\{S\in \calHS: S^*\in \calN(L^2,\mv)\}, 
\ee
and proves that for $S\in \calN^*\bs \{0\}$ the following  result is valid
\ben\label{Eqi-1}
\|\mfV_Sf\|_{\lpqm(\rdd, L^2)}\sim_{S,m,v} \|V_{g_0}f\|_{\lpqm} \ \ \ \text{for all}\  f\in \mpqm,\ 1\leq p,q \leq \fy, m\in \calM_v(\rdd),
\een 
where $\mfV_Sf: =S\pi(z)^*f$.
More precisely, we recall the conclusion in \cite[Theorem 5.1]{Skrettingland2022JoFAaA} as follows.

\textbf{Theorem A.} Let $S\in \calN^*\bs\{0\}$. For any $1\leq p,q\leq \fy$ and $m\in \calM_v(\rdd)$, we have
\ben\label{thm-A-1}
\frac{\|S\|_{\calHS}^2}{C_v^m\|S^*\|_{\calN}\|g_0\|_{\mvd}}\|V_{g_0}f\|_{\lpqm}
\leq
\|\mathfrak{V}_Sf\|_{\lpqm(\rtd; L^2)}\leq C_v^m\|S^*\|_{\calN}\|V_{g_0}f\|_{\lpqm}.
\een

To see the connection between \eqref{Eqi-0} and \eqref{Eqi-1},
we consider a rank-one operator $S_1=\xi\otimes\va$ in $\eqref{Eqi-1}$,
with $\xi\in L^2$ and $\va\in \mv$. Note that $S_1\in \calN^*$ and 
\be
\|\mfV_{S_1}f\|_{L^2}=\|S_1\pi(z)^*f\|_{L^2}=\|\xi\|_{L^2}|V_{\va}f(z)|, \ \ \ \ \|\mfV_{S_1}f\|_{\lpqm(\rdd, L^2)}=\|\xi\|_{L^2}\|V_{\va}f(z)\|_{\lpqm}.
\ee
From this and \eqref{lf}, the equivalent relation \eqref{Eqi-1} can be regarded as an extension for the window class of modulation space, from ``bounded linear functional on $\mvs$'' to ``bounded linear operator from $\mvs$ into $L^2$''.

Comparing with the nice answer for \eqref{Eqi-0}, we naturally ask a corresponding question for \eqref{Eqi-1}, that is, can we give a characterization about the linear operator $S\in \calL(\mvs,L^2)$ 
such that $\eqref{Eqi-1}$ holds?
More precisely,
can we find the precise subset $\calB$ of $\calL(\mvs,L^2)$,
such that $S\in \calB$ if and only if $\eqref{Eqi-1}$ holds?
Note that $\calL(\mvs,L^2)\subset \calL(L^2)$.
In this paper, we will give a complete characterization for \eqref{Eqi-0} in the framework of $\calL(L^2)$, that is, give the precise subset $\calB$ of $\calL(L^2)$ such that $\eqref{Eqi-1}$ holds.
The assumption of $S\in \calL(L^2)$ is convenient for our proofs of main theorems, and the reader will find that this assumption in Theorems \ref{thm-L2} and \ref{thm-M}
can be reduced to a weaker one, that is, $S\in \calL(\mv, L^2)$. See Proposition \ref{pp-ebsb1} and Remark \ref{rk-S1v}
for more details.

First, we deal with the $L^2(\rd)$ case, which yields a new characterization of $\calHS$.
In this case, we only consider the condition \eqref{Eqi-1} with $p=q=2$ and $m=1$.

\begin{theorem}\label{thm-L2}
	Let $S\in \calL(L^2(\rd))\bs \{0\}$. The following four statements are equivalent:
	\begin{enumerate}
		\item 
		$ \|\mathfrak{V}_Sf\|_{L^2(\rtd;L^2)} \sim \|f\|_{L^2(\rd)}$ for all $f\in L^2(\rd)$;
		\item 
		$ \|\mathfrak{V}_Sf\|_{L^2(\rtd;L^2)} \lesssim \|f\|_{L^2(\rd)}$ for all $f\in L^2(\rd)$;
		\item 
		$ \|\mathfrak{V}_Sg_0\|_{L^2(\rtd;L^2)} <\fy$;
		\item $S\in \hs$.
	\end{enumerate}
	Furthermore, if one of the above statements holds, 
	we have 
	\[ \|\mathfrak{V}_Sf\|_{L^2(\rtd;L^2)} =\|S\|_{\hs} \|f\|_{L^2(\rd)},\ \ \ \|S\|_{\hs}=\|\mathfrak{V}_Sg_0\|_{L^2(\rtd;L^2)}.\] 
	If $\|S\|_{\hs}=1$, the map $f\mapsto \mathfrak{V}_Sf$ is an isometry from $L^2(\rd)$ into $L^2(\rdd,L^2)$.
\end{theorem}

As we will see shortly, due to the advantage of Hilbert space, the $L^2$ case is not difficult to deal with.
However, this case is still enlightening.
In fact, in the study of the general modulation case as below,
one can verify $\calB_1^v\subset \hs$ by the logical relationship that the full version of \eqref{Eqi-1} 
is stronger than the special case with $p=q=2$ and $m=1$. See also Proposition \ref{pp-ebsb1} for a sharper conclusion.

Next, we explore the general case. This main theorem can be stated as follows. 
We use $C_v^m$ to denote the constant depending on $v$ and $m$, see Subsection 2.2 for more details.

\begin{theorem}\label{thm-M}
	Let $S\in \calL(L^2(\rd))\bs \{0\}$ and  
	$$\calB_1^v:=\{S\in \calL(L^2(\rd)): \|\mathfrak{V}_Sg_0\|_{\lv(\rtd; L^2)} <\infty\}.$$
	Let $v$ be a submultiplicative weight function on $\rdd$. 
	Denote by $\{e_n\}_{n=1}^{\fy}$ an orthonormal basis of $L^2(\rd)$.
	The following four statements are equivalent:
	\begin{enumerate}
		\item 
		$ \|\mathfrak{V}_Sf\|_{\lpqm(\rtd; L^2)} \sim_{S,m,v} \|f\|_{\mpqmd}$ for all $f\in \mpqmd$, $1\leq p,q \leq \infty$, $m\in \calM_v$;
		\item 
		$ \|\mathfrak{V}_Sf\|_{\lpqm(\rtd; L^2)}  \lesssim_{S,m,v} \|f\|_{\mpqmd}$ for all $f\in \mpqmd$, $1\leq p,q \leq \infty$, $m\in \calM_v$;
		\item 
		$ S\in \calB_1^v$;
		\item 
		$\|(V_{S^* e_n}f)_n\|_{\lpqm(\rdd, l^2)} \sim_{S,m,v} \|f\|_{\mpqmd}$ for all $f\in \mpqmd$, $1\leq p,q \leq \infty$, $m\in \calM_v$.
	\end{enumerate}
	Furthermore, if one of the above statements holds, for $f\in \mpqmd$ we have 
	\ben\label{thm-M-1}
	\frac{\|S\|_{\calHS}^2}{C_v^m\|\mfV_Sg_0\|_{L^1_v(\rdd, L^2)}}\|V_{g_0}f\|_{\lpqm}
	\leq
	\|\mathfrak{V}_Sf\|_{\lpqm(\rtd; L^2)}\leq C_v^m\|\mfV_Sg_0\|_{L^1_v(\rdd,L^2)}\|V_{g_0}f\|_{\lpqm}.
	\een
\end{theorem}

\begin{remark}\label{rmk_B}
The reader may be confused about the definition of $\mathfrak{V}_Sf$ for $S\in \calB_1^v$ in Theorem \ref{thm-M} above.
In fact, by a direct calculation
\be
\begin{split}
	\|V_{g_0}S^*f(z)\|_{L^1_v(\rdd)}
	= &
	\|\lan f, S\pi(z)^*g_0\ran_{L^2}\|_{L^1_v(\rdd)}
	\\
	\leq &
	\|f\|_{L^2}\big\|\|S\pi(z)^*g_0\|_{L^2}\big\|_{L^1_v(\rdd)}
	\\
	= &
	\|f\|_{L^2}\|\mathfrak{V}_Sg_0\|_{\lv(\rtd; L^2)} \lesssim \|f\|_{L^2},
\end{split}
\ee
wee see that $S\in \calB_1^v$ implies that $S^*\in \calL(L^2,\mv)$.
Then the operator $S$ can be naturally extended to a bounded operator from $\mvs$ into $L^2$, also denoted by $S$. Therefore, the operator $\mathfrak{V}_Sf$ is well-defined for all $f\in \mvs$.
For simplicity, we will use $S^*\in \calL(L^2,\mv)$ to denote that $S\in \calL(L^2)$ with its Hilbert adjoint $S^*$ belonging to $\calL(L^2,\mv)$.
Hence, the window class $\calB_1^v$ can be re-represented by
\be
\calB_1^v:=\{S\in \calL(L^2): S^*\in \calL(L^2 ,\mv), \|\mathfrak{V}_Sg_0\|_{\lv(\rtd; L^2)} <\infty\}.
\ee
\end{remark}

\begin{remark}
Comparing with the corresponding result in \cite{Skrettingland2022JoFAaA} (see Theorem A),
the characterization in Theorem \ref{thm-M} is more natural and complete. In our approach,
both the case of bounded linear functional in \eqref{Eqi-0} and the case of bounded linear operator in \eqref{Eqi-1} can be treated in a uniform way, that is,
testing the upper bound inequality in \eqref{Eqi-0} or \eqref{Eqi-1} in the special case of $p=q=1$, $m=v$ and $f=g_0$.
Based on this method, our characterizations are derived directly from the equivalent norm conditions in \eqref{Eqi-1}, without any additional assumptions.
This approach also naturally leads to the corresponding characterization associated with positive Cohen's class distribution, giving an answer for the question posed in \cite[Subsection 7.1]{Skrettingland2022JoFAaA}. 

On the other hand, our style of defining window classes is more convenient for further generalization,
which will be demonstrated in Section \ref{sec-gw}.
\end{remark}

This paper is organized as follows. In Section 2, we collect some basic concepts and properties used in this paper.
Section 3 is devoted to the proofs of our main theorems. The corresponding problems associated with positive Cohen's class distributions are also discussed in Section 3. 
We give a generalization of operator classes in Section \ref{sec-gw}, including some basic properties of general operator classes and some re-exploration of the window class $\calB_1^v$.
Some applications to localization operators are showed at the end of this section.

Throughout this paper, we will adopt the following notations. We use $X\lesssim Y$ to denote the statement $X\leq CY$, with a positive constant $C$ that may depends on $p$, $q$, $d$, but it might be different from line to line. The notation $X\thicksim Y$ means the statement $X\lesssim Y\lesssim X$.
We also use $X\lesssim_{S,m,v} Y$ and $X\sim_{S,m,v} Y$ to denote the similar statements as above with the constant $C$ depending on
$S$, $m$ and $v$.
The inverse of a function is defined by $\tilde{g}(t)=g(-t)$.

\section{Preliminaries}


\subsection{Time-frequency tools}
We consider the point $z=(x,\xi)$ in the time-frequency plane $\rdd$,
where $x, \xi \in \rd$ denote the time  and frequency variables, respectively.
For any fixed $x, \xi$, the translation operator $T_x$, modulation operator $M_{\xi}$ and time-frequency shift $\pi(z)$ are defined, respectively, by
\be
T_xf(t)=f(t-x),\ \ \ \ M_{\xi}f(t)=e^{2\pi it\cdot \xi}f(t),\ \ \ \  \pi(z)f(t)=M_{\xi}T_x  f(t)=e^{2\pi it\cdot\xi}f(t-x).
\ee

The short-time Fourier transform (STFT) of a function $f$ with respect to a window $g$ is defined by
\be
V_gf(x,\xi):=\lan f,\pi(z)g\ran_{L^2},\ \ \  f,g\in L^2(\rd).
\ee
Its extension to $\mvs\times \mv$ can be denoted by
\be
V_gf(x,\xi)=\langle f, \pi(z)g\rangle_{\mvs,\mv},
\ee
in which the STFT $V_gf$ is a bilinear map from $\mvs\times \mv$ into $L^{\fy}_{1/v}$.

A fundamental property we shall use is the following Moyal's identity.
\begin{lemma}\cite[Proposition 4.3.2]{GrochenigBook2013}\label{lm-Moyal}
	Let $f_1,f_2,\va_1,\va_2\in L^2(\rd)$, then $V_{\va_j}f_j\in L^2(\rd)$ for $j=1,2$, satisfy
	\be
	\int_{\rdd}V_{\va_1}f_1(z)\overline{V_{\va_2}f_2(z)}dz=\lan f_1,f_2\ran_{L^2}\overline{\lan \va_1,\va_2\ran}_{L^2}.
	\ee
\end{lemma}

We also need the Fourier transform on a product of STFTs.
\begin{lemma}\cite[Lemma 2.1]{Cordero2003}\label{lm-FTPSTFT}
	If $f_1,f_2,g_1,g_2\in L^2$, we have
	\be
	\scrF\big(V_{g_1}f_1\overline{V_{g_2}f_2}\big)(x,y)=\big(V_{f_2}f_1\overline{V_{g_2}g_1}\big)(-y,x).
	\ee
\end{lemma}

For a non-zero function $\g\in \mv$, we write $V_{\g}^*$ for the adjoint operator of $V_{\g}$, given by
\be
\lan V_{\g}^*F, f\ran=\lan F, V_{\g}f\ran.
\ee
We recall that $V_{\g}^*$ is bounded from $\lpqm$ into $\mpqm$ for $m\in \calM_v$. We also recall the inverse formula as follows.
\begin{lemma} \cite[Theorem 11.3.7]{GrochenigBook2013}
	Assume that $m\in \calM_v$ and let $g,\g\in \mv\bs\{0\}$. Then the following inversion formula is valid
\be
\lan \g, g\ran^{-1}V_{\g}^*V_g=I_{\mvs}.
\ee
\end{lemma}

\subsection{Function spaces}
In order to introduce the function spaces, we first recall some  definitions of weights.
The weights we consider here are the moderate weights, which are suitable for the time-frequency estimates \cite{Groechenig2007POPDEaTA}.
More precisely, a weight function $m$ defined on $\rd$ is called $v$-moderate if there exists another weight function $v$ and a constant $C_v^m$ depending on $v$ and $m$,  such that
\be
m(z_1+z_2)\leq C_v^m v(z_1)m(z_2),\ \ \ \ z_1,z_2\in \rd,
\ee
where $v$ belongs to the class of submultiplicative weight, that is, $v$ satisfies
\be
v(z_1+z_2)\leq v(z_1)v(z_2),\ \ \ \ z_1,z_2\in \rd.
\ee
We use the notation $\calM_v(\rd)$ to denote the cone of all weight functions defined on $\rd$ which are $v$-moderate.
Without loss of generality, we also assume that a $v$-moderate weight is continuous and satisfies
$v(x,\xi)=v(-x,\xi)=v(x,-\xi)=v(-x,-\xi)$. 
We refer to \cite[Lemma 11.2.3]{Heil2008} for more details.

\begin{definition}[Weighted mixed-norm spaces]
	Let $1\leq p,q\leq \fy$, $m\in \calM_v(\rdd)$. Then the weighted mixed-norm space $L^{p,q}_m(\rdd)$
	consists of all Lebesgue measurable functions on $\rdd$ such that the norm
	\be
	\|F\|_{L^{p,q}_m(\rdd)}=\left(\int_{\rd}\left(\int_{\rd}|F(x,\xi) m(x,\xi)|^pdx\right)^{q/p}d\xi\right)^{1/q}
	\ee
	is finite, with the usual modification when $p=\fy$ or $q=\fy$. We write $L^p_m(\rdd)=L^{p,p}_m(\rdd)$ for short.
	If $m\equiv 1$, we write $L^{p,q}(\rdd)=L^{p,q}_m(\rdd)$.
\end{definition}

Now, we introduce the definition of (weighted) modulation space.
\begin{definition}\label{df-M}
	Let $1\leq p,q\leq \infty$, $m\in \calM_v(\rdd)$.
	The (weighted) modulation space $M^{p,q}_m(\rd)$ consists
	of all $f\in \mvs$ such that the norm
	\be
	\begin{split}
		\|f\|_{M^{p,q}_m(\rd)}&:=\|V_{g_0}f\|_{L^{p,q}_m(\rdd)}
		=\left(\int_{\rd}\left(\int_{\rd}|V_{g_0}f(x,\xi)m(x,\xi)|^{p} dx\right)^{{q}/{p}}d\xi\right)^{{1}/{q}}
	\end{split}
	\ee
	is finite, with the usual modification when $p=\fy$ or $q=\fy$.
\end{definition}

We recall a well-known convolution relation of modulation spaces.
\begin{lemma}\cite[Proposition 3.3]{Nicola2009}\label{lm-covm}
	For $p\in [1,\fy]$, we have
	\be
	M^{p,\fy}\ast M^1\subset M^p.
	\ee
\end{lemma}

Among the large classes of modulation spaces, a remarkable one is the Feichtinger algebra $\mv$ that serves as
the admissible window class in the sense of \eqref{Eqi-0}. 
The dual space $\mvs$ can be used as a substitute for the tempered distributions in the general case in which the weight function $v$ grows beyond the polynomial.

\begin{definition}[$L^2$-valued weighted mixed-norm spaces]
For $p,q\in[1,\infty]$ and a $v$-moderate weight $m$, the Banach space $\lpqm(\rtd; L^2)$ consists of all measurable functions $\Psi:\rtd\rightarrow L^2(\rd)$
such that
\[
\|\Psi\|_{\lpqm(\rtd; L^2)}
:=\|\|\Psi(z)\|_{L^2(\rd)}\|_{\lpqmtd}
=\bigg( \int_{\rd} \Big(\int_{\rd} \|\Psi(x,w)\|_{L^2(\rd)}^p m^p(x,w) dx\Big)^{q/p}dw \bigg)^{1/q}
\]
is finite, with the usual modification when $p=\fy$ or $q=\fy$.
\end{definition}

\subsection{Schatten class operator and nuclear operator}
Given a separable Hilbert space $H$ over $\cc$, for $p\in [1,\fy)$,
we use $\sp$ to denote the subspace of $\calL(H)$ consisting of linear compact operators $T$ with
the sequence of singular values belonging to $l^p$, that is, 
\be
\|T\|_{\sp}=\bigg(\sum_{j}s_j(T)^p\bigg)^{1/p}<\fy,
\ee
where $s_j(T)$ denotes the singular values of $T$.
For consistency, we define $\calS_{\fy}=\calL(H)$ to be the space of bounded linear operators on $H$. 

If $p=2$, $\calS_2$ is the space of Hilbert-Schmidt operators, also denoted by $\hs$.
The operator in $\hs$ is called the Hilbert-Schmidt operator, the quantity 
\be
\|T\|_{\hs}=\|T\|_{\calS_2}=\sup\Big\{\Big(\sum_n\|Te_n\|_H^2\Big)^{1/2}:  \{e_n\}\ \text{orthonormal}\Big\}
\ee
is called the Hilbert-Schmidt norm of $T$. 

If $p=1$, $\calS_1$ is the space of trace class operator.
For a trace class operator $T$, we define its trace by
\be
tr(T)=\sum_n \lan Te_n, e_n\ran_H,
\ee
where $\{e_n\}$ is an orthonormal basis of $H$.
In addition, the quantity $\|T\|_{\calS_1}$ is called the trace norm of $T$.

A basic connection between trace class operators and Hilbert-Schmidt operators is that if $S,T\in \calS_2$, then $ST\in \calS_1$.
Specifically, we have $tr(T^*T)=\|T\|_{\calS_2}^2$ for $T\in \calS_2$.

 Next, we recall the nuclear operator mentioned in \cite[Subsection 3.2]{Skrettingland2022JoFAaA}.
 An operator $T\in \calL(L^2,\mv)$ is said to be nuclear if it has an expansion of the form
 \be
 T=\sum_{n=1}^{\fy}\phi_n\otimes \xi_n,
 \ee
with $\sum_{n=1}^{\fy}\|\phi_n\|_{\mv}\|\xi_n\|_{L^2}<\fy$. By $\calN(L^2,\mv)$ we denote the collection of all nuclear operators.
Then $\calN(L^2,\mv)$ becomes a Banach space with the norm given by
\be
\|T\|_{\calN}:= \inf\left\{\sum_{n=1}^{\fy}\|\phi_n\|_{\mv}\|\xi_n\|_{L^2}:\ \  T=\sum_{n=1}^{\fy}\phi_n\otimes \xi_n\right\}.
\ee

\subsection{Khinchin’s inequality}

\begin{lemma}[Khinchin’s inequality, see \cite{Gut2013}]\label{L. Kinchine ineq.}
	Let $0<p<\infty$, $\{\omega_k\}_{k=1}^N$ be a sequence of independent random variables taking values $\pm 1$ with equal probability. Denote expectation (integral over the probability space) by $\mathbb{E}$.
	For any sequence of complex numbers $\{a_k\}_{k=1}^N$, we have
	\begin{equation}
		\mathbb{E}\bigg(\Big|\sum_{k=1}^Na_k\omega_k\Big|^p\bigg)
		\sim 
		\left(\sum_{k=1}^N|a_k|^2\right)^{\frac{p}{2}},
	\end{equation}
	where the implicit constants depend on $p$ only.
\end{lemma}

\section{characterizations of operator window}
\subsection{$L^2$ case}
In this subsection, we deal with the $L^2$ case.
This case reveals us that the suitable window class in \eqref{Eqi-1}  need to be included in the class of Hilbert-Schmidt operators.
\begin{proof}[Proof of Theorem \ref{thm-L2}]
It is obvious that $(1)\Rightarrow (2)\Rightarrow (3)$. Now, we deal with $(3)\Rightarrow (4)$.
Take $\{e_n\}_{n=1}^{\fy}$ to be an orthonormal basis of $L^2(\rd)$. 
By Parseval's identity we have 
\be
\|\mfV_Sg_0(z)\|_{L^2}^2=\sum_{n=1}^{\fy}\big|\lan \mfV_Sg_0(z), e_n\ran_{L^2}\big|^2.
\ee
Note that for $z=(x,\xi)$, we have
\be
\lan \mfV_Sg_0(z), e_n\ran_{L^2}=\lan S\pi(z)^*g_0, e_n\ran_{L^2}=\lan \pi(z)^*g_0, S^*e_n\ran_{L^2}=e^{-2\pi ix\cdot\xi}\lan \pi(-z)g_0, S^*e_n\ran_{L^2}.
\ee
From the above two estimates we have
\ben\label{eq-1}
\|\mfV_sg_0(z)\|_{L^2}^2=\sum_{n=1}^{\fy}|V_{g_0}S^*e_n(-z)|^2,
\een
and 
\be
\begin{split}
	\|\mathfrak{V}_Sg_0\|_{L^2(\rtd;L^2)}^2
	= &
	\int_{\rdd}\sum_{n=1}^{\fy}|V_{g_0}S^*e_n(-z)|^2dz\\
	= &
	\sum_{n=1}^{\fy}\int_{\rdd}|V_{g_0}S^*e_n(z)|^2dz=\sum_{n=1}^{\fy}\|S^*e_n\|_{L^2}^2,
\end{split}
\ee
where in the last equality we use Moyal's identity (Lemma \ref{lm-Moyal}). From this and the assumption (3), we conclude that
\be
\|S^*\|_{\calHS}=\big(\sum_{n=1}^{\fy}\|S^*e_n\|_{L^2}^2\big)^{1/2}=\|\mathfrak{V}_Sg_0\|_{L^2(\rtd;L^2)}<\fy,
\ee
which yields that $S^*\in \calHS$. Then, we obtain $S\in \calHS$ with $\|S\|_{\calHS}=\|S^*\|_{\calHS}$.

Finally, we consider $(4)\Rightarrow (1)$. Using Parseval's identity and the fact 
\be
\lan \mfV_Sf(z), e_n\ran_{L^2}=\lan f, \pi(z)S^* e_n\ran_{L^2}=V_{S^*e_n}f(z),
\ee 
we have
\ben\label{eq-3}
\|\mfV_Sf(z)\|_{L^2}^2=\sum_{n=1}^{\fy}\big|\lan \mfV_Sf(z), e_n\ran_{L^2}\big|^2=\sum_{n=1}^{\fy}|V_{S^*e_n}f(z)|^2.
\een
Then, we conclude (1) by
\be
\begin{split}
	\|\mathfrak{V}_Sf\|_{L^2(\rtd;L^2)}^2
	= &
	\int_{\rdd}\sum_{n=1}^{\fy}|V_{S^*e_n}f(z)|^2dz\\
	= &
	\sum_{n=1}^{\fy}\|S^*e_n\|_{L^2}^2\|f\|_{L^2}^2=\|S\|_{\calHS}^2\|f\|_{L^2}^2,
\end{split}
\ee
where in the last second equality we use Moyal's identity.
\end{proof}
It should not be difficult to see that Theorem \ref{thm-L2} and its proof are still valid 
when $L^2(\rd)$ is replaced by any separable Hilbert space $H$.

\subsection{$\mpqm$ case}

In order to deal with the general modulation space $\mpqm$,
 we first recall the following pointwise inequality of $STFT$.
One can find the following result from Lemma 11.3.3 in \cite{GrochenigBook2013}.

\begin{lemma}\label{lm-cSTFT}
	Let $\va,\psi\in \mv$, $f\in \mvs$. We have the following inequality
	\be
	|V_{\va}f|\leq \|\psi\|_{L^2}^{-2}|V_{\va}\psi|\ast |V_{\psi}f|.
	\ee
\end{lemma}

Using a randomization technique, we establish the following vector-valued inequality.
\begin{proposition}\label{pp-cwid}
	Let $1\leq p,q\leq \fy$,
	$m\in \calM_v$ and $(\va_n)_{n=1}^\fy\subset \mv$.
	We have the following pointwise inequality
	\ben\label{pp-cwid-1}
	\Big( \sum_{n=1}^N \big| V_{\va_n}f \big|^2  \Big)^{1/2}
	\lesssim \Big( \sum_{n=1}^N \big| V_{\va_n}g_0 \big|^2  \Big)^{1/2}\ast |V_{g_0}f|.
	\een
	Moreover, if $\|(V_{\va_n}g_0)_n\|_{L^1_v(\rdd,l^2)}<\fy$, then the map $f \mapsto (V_{\va_n}f)_{n=1}^{\fy}$ is bounded from $\mpqm(\rd)$ to $\lpqm(\rdd,l^2)$ with
	\ben\label{pp-cwid-2}
	\|(V_{\va_n}f)_{n}\|_{\lpqm(\rdd,l^2)}\lesssim C_v^m\|(V_{\va_n}g_0)_n\|_{L^1_v(\rdd,l^2)}\|V_{g_0}f\|_{\lpqm}.
	\een
\end{proposition}
\begin{proof}
Let $r_n(t)$ be a sequence of independent random variables taking values $\pm 1$ with equal probability.
Using Lemma \ref{lm-cSTFT}, we have
	\be
	|\sum\limits_{n=1}^N r_n(t)V_{\va_n}f|
	=
|V_{\sum\limits_{n=1}^N r_n(t)\va_n}f|
\leq
|V_{\sum\limits_{n=1}^Nr_n(t)\va_n}g_0| \ast |V_{g_0}f|.
\ee
Taking expectation on both sides and using the Khinchin inequality, we obtain
\be
\begin{split}
\Big( \sum_{n=1}^N \big| V_{\va_n}f \big|^2  \Big)^{1/2} 
\sim &
\mathbb{E}\Big(\Big|\sum\limits_{n=1}^N r_n(t)V_{\va_n}f\Big|\Big)
\\
\leq &
\mathbb{E}\Big(\Big|V_{\sum\limits_{n=1}^N r_n(t)\va_n}g_0\Big| \ast |V_{g_0}f|\Big)
\\
= &
\mathbb{E}\Big(\Big|\sum_{n=1}^Nr_n(t)V_{\va_n}g_0\Big|\Big) \ast |V_{g_0}f|
\sim\Big( \sum_{n=1}^N \big| V_{\va_n}g_0 \big|^2  \Big)^{1/2}\ast |V_{g_0}f|.
\end{split}
\ee
Applying the convolution inequality $\lpqm\ast L^1_v\subset \lpqm$ and letting $N\rightarrow\infty$, 
we obtain \eqref{pp-cwid-2}.

\end{proof}

\begin{proposition}\label{pp-cwi}
	Let $1\leq p,q\leq \fy$, $m\in \calM_v$ and $S\in \calB_1^v$. We have the following pointwise inequality for all $z\in \rdd$
	\ben\label{pp-cwi-1}
	\|\mfV_Sf(z)\|_{L^2}\lesssim  \left(\|\mfV_Sg_0(\cdot)\|_{L^2}\ast |V_{g_0}f(\cdot)|\right)(z).
	\een
	Moreover, the map $f \mapsto \mfV_Sf$ is bounded from $\mpqm(\rd)$ to $\lpqm(\rdd,L^2)$ with
	\ben\label{pp-cwi-2}
\|\mfV_Sf\|_{\lpqm(\rdd,L^2)}\lesssim C_v^m\|\mfV_Sg_0\|_{L^1_v(\rdd,L^2)}\|V_{g_0}f\|_{\lpqm(\rdd)}.
\een
\end{proposition}
\begin{proof}
Recall that $S\in \calB_1^v$ implies $S^*\in \calL(L^2, \mv)$, then $S$ can be extended to a bounded operator from $\mvs$ into $L^2$, also denoted by $S$.
Take $\{e_n\}_{n=1}^{\fy}$ to be an orthonormal basis of $L^2(\rd)$.
Note that for $f\in \mvs$, $z=(x,\xi)$,
\be
\lan \mfV_Sf(z), e_n\ran_{L^2}=\lan S\pi(z)^*f, e_n\ran_{L^2}=\lan \pi(z)^*f, S^*e_n\ran_{\mvs,\mv}=\lan f, \pi(z)S^*e_n\ran_{\mvs,\mv},
\ee
where $S^*e_n\in \mv$.
We have
\be
\|\mfV_Sf(z)\|_{L^2}^2
=\sum_{n=1}^{\fy}\big|\lan f, \pi(z)S^*e_n\ran_{\mvs,\mv}\big|^2
=\sum_{n=1}^{\fy}|V_{S^*e_n}f(z)|^2.
\ee
Using this and Proposition \ref{pp-cwid}, we conclude that
\be
\begin{split}
	\|\mfV_Sf(z)\|_{L^2}
	= &
	\|(V_{S^* e_n}f(z))_n\|_{l^2}
	\\
	\lesssim &
	\left\|(V_{S^*e_n}g_0(\cdot))_n\|_{l^2}\ast |V_{g_0}f(\cdot)|\right)(z)
    =
	\left(\|\mfV_Sg_0(\cdot)\|_{L^2}\ast |V_{g_0}f(\cdot)|\right)(z).
\end{split}
\ee
Finally, \eqref{pp-cwi-2} follows by the convolution inequality $\lpqm\ast L^1_v\subset \lpqm$.
\end{proof}

In order to obtain the lower bound estimate of $\|\mfV_Sf\|_{\lpqm(\rdd,L^2)}$, we
need a $\mvs$ reconstruction associated with $\mfV_S$.
We establish this reconstruction by the following classical method.
A similar process has also been carried out in \cite{Skrettingland2022JoFAaA}. We recall the operator $\mfV_S^*$ for $S\in \calB_1^v$ as follows.
\be
\lan \mfV_S^*F, \va\ran_{\mvs,\mv}:= \int_{\rdd}\lan F(z), \mfV_S\va\ran_{L^2} dz,\ \ \ F\in \lpqm(\rdd, L^2),\ \va\in \mv(\rd),
\ee
where the right term leads to a bounded linear functional on $\mv(\rd)$.
For the boundedness of $\mfV_S^*$ we recall the following lemma (see \cite[Lemma 5.3]{Skrettingland2022JoFAaA}) with slight modification.
\begin{lemma}\label{lm-cov}
Let $S\in \calB_1^v$, $m\in \calM_v$. For $1\leq p,q\leq \fy$, the map $\mfV_S^*$ is bounded from $\lpqm(\rdd,L^2)$ into $\mpqm(\rd)$ with the following inequality
\be
\|V_{g_0}\mfV_S^*F\|_{\lpqm(\rdd)}\leq  C_v^m\|F\|_{\lpqm(\rdd,L^2)}\|\mfV_Sg_0\|_{L^1_v(\rdd, L^2)}.
\ee
\end{lemma}

Next, we turn to the reconstruction on $\mvs$. First, we recall a useful result in \cite[Lemma 4.1]{LuefSkrettingland2018JdMPeA}.
\begin{lemma}\label{scov}
Let $R, T\in \calS_1(L^2)$ be trace class operators. Then the function $z\mapsto tr(\pi(z)R\pi(z)^*T)$ is integrable with $\|tr(\pi(z)R\pi(z)^*T)\|_{L^1}\leq \|R\|_{\calS_1(L^2)}\|T\|_{\calS_1(L^2)}$. Furthermore,
\be
\int_{\rdd}tr(\pi(z)R\pi(z)^*T)dz=tr(R)tr(T).
\ee
\end{lemma}

Now, we give the reconstruction on $\mvs$ by the method of \cite[Lemma 5.4]{Skrettingland2022JoFAaA} with slight modification.
\begin{proposition}\label{pp-rc}
	Let $S,T\in \calB_1^v$. We have $\mfV_T^*\mfV_S=tr(T^*S)I_{\mvs}$. Specifically, we have
	\be
	\mfV_S^*\mfV_S=tr(S^*S)I_{\mvs}=\|S\|_{\calS_2}^2I_{\mvs}.
	\ee
\end{proposition}
\begin{proof}
	We need to verify that
	\ben\label{pp-rc-1}
	\lan \mfV_T^*\mfV_S f,\va\ran_{\mvs,\mv}=tr(T^*S)\lan f, \va\ran_{\mvs,\mv},
	\een
	for all $f\in \mvs$ and $\va\in \mv$.
	This identity is valid for $f\in L^2$, since
	\be
	\begin{split}
		\lan \mfV_T^*\mfV_S f,\va\ran_{\mvs,\mv}
		= &
		\int_{\rdd}\lan \mfV_Sf, \mfV_T\va \ran_{L^2}dz
		\\
		= &
		\int_{\rdd}\lan \pi(z)T^*S\pi(z)^*f, \va \ran_{L^2}dz
		\\
		= &
		\int_{\rdd}tr((\pi(z)T^*S\pi(z)^*f)\otimes \va)dz
		\\
		= &
		\int_{\rdd}tr(\pi(z)T^*S\pi(z)^*(f\otimes \va))dz,
	\end{split}
	\ee
	where by Lemma \ref{scov} the last term equals to 
	\be
	tr(T^*S)tr(f\otimes \va)=tr(T^*S)\lan f,\va\ran_{L^2}=tr(T^*S)\lan f,\va\ran_{\mvs,\mv}.
	\ee
	For $f\in \mvs$, recall that $S\in \calB_1^v$ implies $S^*\in \calL(L^2, \mv)$, and write
		\be
	\begin{split}
		\lan \mfV_T^*\mfV_S f,\va\ran_{\mvs,\mv}
			= &
		\int_{\rdd}\lan \mfV_Sf, \mfV_T\va \ran_{L^2}dz
		\\
		= &
		\int_{\rdd}\lan f, \pi(z)S^*T\pi(z)^*\va \ran_{\mvs,\mv}dz.
	\end{split}
	\ee
	Then, \eqref{pp-rc-1} is equivalent to 
	\ben\label{pp-rc-2}
	\int_{\rdd}\lan f, \pi(z)S^*T\pi(z)^*\va \ran_{\mvs,\mv}dz=tr(T^*S)\lan f, \va\ran_{\mvs,\mv},
	\een
	which has been verified for $f\in L^2(\rd)$. For $f\in \mvs$, there exists a sequence $\{f_n\}_{n=1}^{\fy}\subset L^2(\rd)$ that tends to $f$ in the weak* topology of $\mvs$, and satisfies $\|f_n\|_{\mvs}\lesssim \|f\|_{\mvs}$.
	Then, by \eqref{pp-rc-2} we obtain
	\ben\label{pp-rc-3}
	\int_{\rdd}\lan f_n, \pi(z)S^*T\pi(z)^*\va \ran_{\mvs,\mv}dz=tr(T^*S)\lan f_n, \va\ran_{\mvs,\mv},
	\een
	where the right term tends to $tr(T^*S)\lan f, \va\ran_{\mvs,\mv}$ as $n\rightarrow \fy$. The remaining issue is to deal with the left term by letting $n\rightarrow \fy$.
	For the sequence of functions $z\mapsto \lan f_n, \pi(z)S^*T\pi(z)^*\va \ran_{\mvs,\mv}$ that tends to $\lan f, \pi(z)S^*T\pi(z)^*\va \ran_{\mvs,\mv}$ as $n\rightarrow \fy$, we find the dominated function by
	\be
	\begin{split}
		|\lan f_n, \pi(z)S^*T\pi(z)^*\va \ran_{\mvs,\mv}|
		= &
		|\lan \mfV_Sf_n, \mfV_T\va \ran_{L^2}|
		\\
		\leq &
		\|\mfV_Sf_n\|_{L^2}\|\mfV_T\va\|_{L^2}
		\\
		= &
		v(z)^{-1}\|\mfV_Sf_n\|_{L^2}v(z)\|\mfV_T\va\|_{L^2}
		\\
		\leq &
		\big\|\|\mfV_Sf_n\|_{L^2}\big\|_{L^\fy_{1/v}}\|\mfV_T\va\|_{L^2}v(z)
		\\
		\lesssim & 
		\|f\|_{\mvs}\|\mfV_T\va\|_{L^2}v(z)\in L^1(\rdd),
	\end{split}
	\ee
	where in the last inequality we use Proposition \ref{pp-cwi} and $\|f_n\|_{\mvs}\lesssim \|f\|_{\mvs}$. Then the Lebesgue dominated convergence theorem yields that 
	the left term in \eqref{pp-rc-3} tends to $\int_{\rdd}\lan f, \pi(z)S^*T\pi(z)^*\va \ran_{\mvs,\mv}dz$. We have now completed this proof.
\end{proof}

\begin{proof}[Proof of Theorem \ref{thm-M}]
The proof of $(1)\Rightarrow (2)$ is obvious, and the relation $(2)\Rightarrow (3)$ follows by taking $p=q=1$, $f=g_0$ and $m=v$. 
Now, we consider the relation $(3)\Rightarrow (1)$.
The upper bound follows by Proposition \ref{pp-cwi}. For the lower bound, we use Proposition \ref{pp-rc} and Lemma \ref{lm-cov} to deduce that
\be
\begin{split}
	\|f\|_{\mpqm}
	= 
	\frac{1}{\|S\|_{\calHS}^2}\|\mfV_S^*\mfV_Sf\|_{\mpqm}
	\leq 
	\frac{C_v^m\|\mfV_Sg_0\|_{L^1_v(\rdd, L^2)}}{\|S\|_{\calHS}^2}\|\mfV_Sf\|_{\lpqm(\rdd,L^2)}.
\end{split}
\ee
Finally, $(1)\Longleftrightarrow (4)$ follows by \eqref{eq-3}.

\end{proof}

\subsection{Positive Cohen's class distributions}
In this subsection, we focus on the reinterpretation of Theorems \ref{thm-L2} and \ref{thm-M} by using Cohen's class distribution.
We refer to \cite[Section 7]{Skrettingland2022JoFAaA} for the corresponding research on this topic. In some sense, we give an answer to
the question posed in \cite[Subsection 7.1]{Skrettingland2022JoFAaA}. See also \cite[Example 7.3]{Skrettingland2022JoFAaA} for a discussion
on the positive assumption of $T$.

Given a Hilbert-Schmidt operator $T$ on $L^2$, the Cohen's class distribution $Q_T$ can be defined on $L^2$ by
\be
Q_Tf(z)=\lan T\pi(z)^*f, \pi(z)^*f\ran_{L^2},\ \ \ f\in L^2.
\ee
This definition was given in \cite{Luef2019}.
It can be regarded as a generalization of the classical Cohen's class distribution defined by $Q_a(f)=a\ast W(f)$ for $a,f\in L^2$. 
Here,  $W(f)$ denotes the Wigner distribution of $f$.
Using this generalized definition of Cohen's class distribution,
 we firstly give the following conclusion corresponding to Theorem \ref{thm-L2}.

\begin{theorem}\label{thm-L2c}
	Let $T\in \calL(L^2(\rd))\bs \{0\}$ be a positive operator. 
	Denote by $\sqrt{T}$ the positive square of $T$.
	The following five statements are equivalent:
	\begin{enumerate}
		\item 
		$ \|\sqrt{Q_Tf}\|_{L^2(\rdd)} \sim \|f\|_{L^2(\rd)}$ for all $f\in L^2(\rd)$;
		\item 
		$ \|\sqrt{Q_Tf}\|_{L^2(\rdd)} \lesssim \|f\|_{L^2(\rd)}$ for all $f\in L^2(\rd)$;
		\item 
		$ \|\sqrt{Q_Tg_0}\|_{L^2(\rdd)} < \fy$;
		\item $\sqrt{T}\in \calS_2$;
		\item $T$ is a trace class operator.
	\end{enumerate}
	Furthermore, if one of the above statements holds, we have 
	\[ \|\sqrt{Q_Tf}\|_{L^2(\rdd)} =\sqrt{tr(T)}\|f\|_{L^2},\ \ \ tr(T)=\|\sqrt{Q_Tg_0}\|_{L^2(\rdd)}^2.\] 
\end{theorem}
\begin{proof}
	For $f\in L^2$, write
	\be
	Q_Tf=\lan T\pi(z)^*f,\pi(z)^*f\ran_{L^2}=\lan \sqrt{T}\pi(z)^*f,\sqrt{T}\pi(z)^*f\ran_{L^2}.
	\ee
	Denote $S=\sqrt{T}$, we have
	\be
	\sqrt{Q_Tf}=\|\mathfrak{V}_Sf\|_{L^2}\ \ \text{and}\ \ \  \|\mathfrak{V}_Sf\|_{L^2(\rtd;L^2)}=\|\sqrt{Q_Tf}\|_{L^2(\rdd)}.
	\ee
	Then, the desired equivalent relations follows by Theorem \ref{thm-L2} and the fact that $\|S\|_{\calS_2}=\sqrt{tr(T)}$.
\end{proof}

Next, we explore the corresponding result of Theorem \ref{thm-M}. 
We use $T^*\in \calL(\mvs,\mv)$ to denote that $T\in \calL(L^2)$ with its Hilbert adjoint $T^*$ belonging to $\calL(\mvs,\mv)$.
For $T^*\in \calL(\mvs,\mv)$, Cohen's class distribution associated with $T$ can be defined on $\mvs$ by
\ben\label{df-C}
Q_Tf(z)=\lan \pi(z)^*f, T^*\pi(z)^*f\ran_{\mvs,\mv},\ \ \ \ f\in \mvs.
\een
See \cite[Remark 6]{Skrettingland2022JoFAaA} for more details of the definition of Cohen's class distribution.
Now, we give the following conclusion corresponding to Theorem \ref{thm-M}.

\begin{theorem}\label{thm-Mc}
	Let $T\in \calL(L^2(\rd))\bs \{0\}$ be a positive operator.
	Let $v$ be a submultiplicative weight function on $\rdd$. 
	Denote by $\widetilde{\mpqmd}$ the $L^2(\rd)$ closure in $\mpqmd$.
	The following statements are equivalent:
	\begin{enumerate}
		\item 
		$ \|\sqrt{Q_Tf}\|_{\lpqm(\rtd)} \sim_{S,m,v} \|f\|_{\mpqmd}$ for all $f\in \widetilde{\mpqmd}$, $1\leq p,q \leq \infty$, $m\in \calM_v$;
		\item 
		$ \|\sqrt{Q_Tf}\|_{\lpqm(\rtd)}  \lesssim_{S,m,v} \|f\|_{\mpqmd}$ for all $f\in \widetilde{\mpqmd}$, $1\leq p,q \leq \infty$, $m\in \calM_v$;
		\item 
		$ \sqrt{T}\in \calB_1^v$.
	\end{enumerate}
	Furthermore, if one of the above statements holds, for $f\in \widetilde{\mpqmd}$  we have 
	\ben\label{thm-Mc-1}
	\frac{tr(T)}{C_v^m\|\sqrt{Q_Tg_0}\|_{L^1_v(\rdd)}}\|V_{g_0}f\|_{\lpqm}
	\leq
	\|\sqrt{Q_Tf}\|_{\lpqm(\rtd)}
	\leq C_v^m\|\sqrt{Q_Tg_0}\|_{\lv(\rtd)} \|V_{g_0}f\|_{\lpqm}.
	\een
\end{theorem}
\begin{proof}
	The relation  $(1)\Rightarrow (2)$ is obvious, and the relation $(2)\Rightarrow (3)$ follows by taking $p=q=1$, $f=g_0$ and $m=v$.  
	Then we have
	\be
	\begin{split}
	\|\mathfrak{V}_{\sqrt{T}}g_0\|_{\lv(\rtd; L^2)}
	= &
	\|\sqrt{\lan \sqrt{T}\pi(z)^*g_0,  \sqrt{T}\pi(z)^*g_0 \ran}\|_{L^1_v(\rdd)}
	\\
	= &
	\|\sqrt{\lan T\pi(z)^*g_0,  \pi(z)^*g_0 \ran}\|_{L^1_v(\rdd)}
	=\|\sqrt{Q_Tg_0}\|_{L^1_v(\rdd)}
	\lesssim 
	\|g_0\|_{\mv(\rd)}<\fy.
	\end{split}
   \ee
	
	Next, we only need to verify the inverse direction 
	$(3)\Rightarrow (2)\Rightarrow (1)$.
	As in the proof of Theorem \ref{thm-L2c}, for $f\in L^2$ and $S=\sqrt{T}$,  we have
	\ben\label{3}
	\sqrt{Q_Tf}=\|\mathfrak{V}_Sf\|_{L^2},\ \ \ 
	\|\mathfrak{V}_Sf\|_{\lpqm(\rtd; L^2)}=\|\sqrt{Q_Tf}\|_{\lpqm(\rtd)}.
	\een
	Using this with Theorem \ref{thm-M} and the fact $S=\sqrt{T}\in \calB_1^v$,
    we conclude that the statement (2) is valid for $f\in L^2$.
	For $f\in \widetilde{\mpqm}$, there exists a sequence of $L^2$ functions denoted by $\{f_j\}_{j=1}^{\fy}$ such that $f_j\rightarrow f$ in the topology of $\mpqm$.
	Since $f_j\in L^2$, we have
	\be
	\|\sqrt{Q_Tf_j}\|_{\lpqm(\rtd)}  \lesssim_{S,m,v} \|f_j\|_{\mpqmd}.
	\ee
	Recalling $\sqrt{T}\in \calB_1^v$, we conclude $\sqrt{T}\in \calL(L^2,\mv)$ by Remark \ref{rmk_B}.
By this and that $\sqrt{T}$ is self-adjoint,  the operator $S=\sqrt{T}\in \calL(L^2)$ can be extended by duality to be a bounded operator from $\mvs$ into $L^2$, also denoted by $\sqrt{T}$. Using this extension of $\sqrt{T}$, the operator $T=\sqrt{T}\sqrt{T}$ can be naturally extended to be a bounded operator from $\mvs$ into $\mv$.
	Since $T$ is self-adjoint, $T^*$ is also extended automatically in this way.
	
	Using the fact that $T^*\in \calL(\mvs,\mv)$,
	we conclude that  $T^*\pi(z)^*f_j$ tends to $T^*\pi(z)^*f$ in $\mv$.
	From this and the continuity of the bilinear map $\lan \cdot,\cdot\ran_{\mvs,\mv}$, we obtain 
	\ben\label{4}
	Q_Tf_j(z)=\lan \pi(z)^*f_j,T^*\pi(z)^*f_j\ran_{\mvs,\mv}\rightarrow \lan \pi(z)^*f,T^*\pi(z)^*f\ran_{\mvs,\mv}=Q_Tf(z),
	\een
	where the convergence process is valid for each point $z\in \rdd$ as $j\rightarrow \fy$. By applying Fatou's lemma, we conclude that
	\be
	\begin{split}
	\|\sqrt{Q_Tf}\|_{\lpqm(\rtd)}
	= &
	\|\liminf_{j\rightarrow \fy}\sqrt{Q_Tf_j}\|_{\lpqm(\rtd)}
	\\
	\leq &
	\liminf_{j\rightarrow \fy}\|\sqrt{Q_Tf_j}\|_{\lpqm(\rtd)}
	\\
	\lesssim&_{S,m,v} 
	\liminf_{j\rightarrow \fy}\|f_j\|_{\mpqmd}=\|f\|_{\mpqmd}.
	\end{split}
	\ee
	This completes the proof of statement $(3)\Rightarrow (2)$. 
	
	Next, we turn to the proof $(2)\Rightarrow (1)$. Using Theorem \ref{thm-M} with the fact
	\eqref{3}, we obtain
	\be
	\|\sqrt{Q_Tf_j}\|_{\lpqm(\rtd)} \sim_{S,m,v} \|f_j\|_{\mpqmd},
	\ee
	where $f_j$ is the approximating sequence mentioned above. 
	We claim that 
	\be
	\|\sqrt{Q_Tf_j}\|_{\lpqm(\rtd)}\rightarrow \|\sqrt{Q_Tf}\|_{\lpqm(\rtd)}\ \ \ (j\rightarrow \fy),
	\ee
	then the desired conclusion follows by this claim and the fact that $\|f_j\|_{\mpqmd}\rightarrow \|f\|_{\mpqmd}$ as $j\rightarrow \fy$.
	
	Now, we verify the claim. Using the fact in \eqref{3}, we conclude that
	\be
	\begin{split}
	|\sqrt{Q_Tf_j}-\sqrt{Q_Tf_l}|
	= &
	|\|\mathfrak{V}_Sf_j\|_{L^2}-\|\mathfrak{V}_Sf_l\|_{L^2}|
	\\
	\leq &
	\|\mathfrak{V}_S(f_j-f_l)\|_{L^2}=\sqrt{Q_T(f_j-f_l)}.
	\end{split}
	\ee
	Letting $l\rightarrow \fy$ and using \eqref{4}, we conclude that
	\be
	|\sqrt{Q_Tf_j}-\sqrt{Q_Tf}|\leq \sqrt{Q_T(f_j-f)}.
	\ee
	Taking the $\lpqm$ norm on both sides, we have
	\be
	\|\sqrt{Q_Tf_j}-\sqrt{Q_Tf}\|_{\lpqmtd}
	\leq 
	\|\sqrt{Q_T(f_j-f)}\|_{\lpqm(\rtd)}.
	\ee
	By the conclusion in statement (2), we obtain that
	$\|\sqrt{Q_T(f_j-f)}\|_{\lpqm(\rtd)}  \lesssim_{S,m,v} \|f_j-f\|_{\mpqmd}$.
    The claim follows by
	\be
	\begin{split}
		\big|\|\sqrt{Q_Tf_j}\|_{\lpqmtd}-\|\sqrt{Q_Tf}\|_{\lpqmtd}\big|
		\leq &
		\|\sqrt{Q_Tf_j}-\sqrt{Q_Tf}\|_{\lpqmtd}
		\\
		\leq &
		\|\sqrt{Q_T(f_j-f)}\|_{\lpqm(\rtd)}
		\lesssim 
		\|f_j-f\|_{\mpqmd},
	\end{split}
	\ee
	where the last term tends to zero as $j\rightarrow \fy$.
	
	Finally, if one of the statements $(1-3)$ is valid, 
	by using Theorem \ref{thm-M} and the fact $\|\sqrt{T}\|_{\calS_2}=\sqrt{tr(T)}$,
	we conclude that \eqref{thm-Mc-1} is valid for $f\in L^2$.
	Then the desired conclusion follows by a similar limiting argument used above.
\end{proof}
\begin{remark}
We point out that in Theorem \ref{thm-Mc}, the space $\widetilde{\mpqm}$ can be replaced by $\mpqm$ when $p,q<\fy$,
since $\widetilde{\mpqm}=\mpqm$ in this case.
\end{remark}

\begin{remark}
	In the proof of Theorem \ref{thm-Mc}, one can find that $\sqrt{T}\in \calB_1^v$ implies $T^*\in \calL(\mvs,\mv)$.
	Therefore, Cohen's class distribution $Q_Tf$ in \eqref{df-C} can be defined for $f\in \mvs$.
\end{remark}

\subsection{The relation between $\calB_1^v$ and $\calN^*$}

\begin{proposition}\label{pp-ebns}
	The following embedding relations is valid
	\be
	\calN^*\subset \calB_1^v.
	\ee
\end{proposition}
\begin{proof}
	Although this conclusion is implied in the logical relationship by
	\be
	S\in \calN^* \Rightarrow \eqref{Eqi-1} \Leftrightarrow S\in \calB_1^v,
	\ee
	we would like to give a direct proof here. Let $S\in \calN^*$, then 
	\be
	S=\sum_{n=1}^{\fy} \xi_n\otimes \phi_n\ \ \text{with}\ \ \sum_{n=1}^{\fy}\|\xi_n\|_{L^2}\|\phi_n\|_{\mv}<\fy.
	\ee
	By a direct calculation, we have
	\be
	\begin{split}
		\|\mathfrak{V}_Sg_0\|_{L^2}
		=
		\bigg\|\sum_{n=1}^{\fy} \xi_n\lan \pi(z)^*g_0,\phi_n\ran_{L^2}\bigg\|_{L^2}
		\leq
		\sum_{n=1}^{\fy} \|\xi_n\|_{L^2}|\lan \pi(z)^*g_0,\phi_n\ran_{L^2}|.
	\end{split}
	\ee
	Then,
	\be
	\begin{split}
	\|\mathfrak{V}_Sg_0\|_{\lv(\rtd; L^2)} 
	\leq &
	\sum_{n=1}^{\fy} \|\xi_n\|_{L^2}\|\lan \pi(z)^*g_0,\phi_n\ran_{L^2}\|_{L^1_v}
	\\
	= &
	\sum_{n=1}^{\fy} \|\xi_n\|_{L^2}\|V_{g_0}\phi_n\|_{L^1_v}=\sum_{n=1}^{\fy} \|\xi_n\|_{L^2}\|\phi_n\|_{\mv}.
	\end{split}
	\ee
	We have now completed this proof.
	Moreover, we obtain that $\|\mathfrak{V}_Sg_0\|_{\lv(\rtd; L^2)} \lesssim \|S^*\|_{\calN(L^2,\mv)}$
	by the definition of $\calN(L^2,\mv)$. 
\end{proof}
\begin{remark}
	As we see, the window class $\calB_1^v$ give a complete characterization of  \eqref{Eqi-1},
	while Proposition \ref{pp-ebns} tells us that $\calN^*$ is a subspace of $\calB_1^v$.
	However, we are still confused about whether $\calN^*$ is a proper subset of $\calB_1^v$.
\end{remark}

\section{The $\calB_{p,q}^m$ operator classes}\label{sec-gw}
As mentioned above, the $\calB_1^v$ class is the optimal window class in the framework of bounded operators on $L^2$.
Here, we will introduce some more general classes of operators that may be of independent interest.

\subsection{Start with the $\calB_{\fy}^{v^{-1}}$ class}
In order to introduce our general classes of operators, we would like to start with the weakest assumption.
Let 
\be
\calH_0=\text{span}\{\pi(z)g_0: z\in \rdd\}
\ee
be the linear space of all finite linear combination of time-frequency shifts of the Guassian function $g_0$.
For a submultiplicative weight $v$, we use $\calB_{\fy}^{v^{-1}}$ to denote the collection of linear operators defined on $\calH_0$ satisfying
\be
\|T\pi(z)^*g_0\|_{L^{2}}\leq C_{\calB_{\fy}^{v^{-1}}}v(z)
\ee
for all $T\in \calB_{\fy}^{v^{-1}}$ and $z\in \rdd$, where $C_{\calB_{\fy}^{v^{-1}}}=\Big\|\|T\pi(z)^*g_0\|_{L^{2}}\Big\|_{L^{\fy}_{v^{-1}}(\rdd)}$. We claim that
\be
\calB_{\fy}^{v^{-1}}=\calL(\mv, L^2).
\ee
We first verify $\calL(\mv, L^2)\subset \calB_{\fy}^{v^{-1}}$ by
\be
\|T\pi(z)^*g_0\|_{L^{2}}=\|T\pi(-z)g_0\|_{L^{2}}\leq \|T\|_{\calL(\mv, L^2)}\|\pi(-z)g_0\|_{\mv}
\leq
v(z)\|T\|_{\calL(\mv, L^2)}\|g_0\|_{\mv}.
\ee
This also implies $T\in \calB_1^v$ with
\be
 C_{\calB_{\fy}^{v^{-1}}}\lesssim \|T\|_{\calL(\mv, L^2)}.
\ee
On the other hand, if $T\in \calB_{\fy}^{v^{-1}}$, for any Gabor expansion
\be
f=\sum_{n=1}^{\fy}c_n\pi(z_n)g_0
\ee
with $\sum_{n=1}^{\fy}|c_n|v(z_n)<\fy$, we have
\be
\|T(\sum_{n=1}^{N}c_n\pi(z_n)g_0)\|_{L^2}\leq \sum_{n=1}^{N}|c_n|\|T(\pi(z_n)g_0)\|_{L^2}
\leq
C_{\calB_{\fy}^{v^{-1}}} \sum_{n=1}^{N}|c_n|v(z_n).
\ee
Then, the operator $T\in \calB_{\fy}^{v^{-1}}$, first defined on $\calH_0$ can be uniquely extended to a bounded linear operator 
from $\mv$ into $L^2$ with 
\be
\|T\|_{\calL(\mv,L^2)}\lesssim C_{\calB_{\fy}^{v^{-1}}}.
\ee 
The claim is proved. 
Now, we take $\calB_{\fy}^{v^{-1}}$ as the largest operator class in our discussion, 
just like the status of $M^{\fy}_{v^{-1}}$ in the class of modulation spaces (with $p,q\in [1,\fy]$).

\subsection{The $\calB_{p,q}^m$ classes and their relations}
With $p,q\in [1,\fy]$ and $m\in \calM_v$, the $\calB_{p,q}^m$ is defined as
\be
\calB_{p,q}^m=\{T\in \calB_{\fy}^{v^{-1}}: \Big\|\|T\pi(z)^*g_0\|_{L^2}\Big\|_{\lpqm(\rdd)}< \fy\}
\ee
with the obvious norm. Since $\calB_{\fy}^{v^{-1}}=\calL(\mv, L^2)$, the $\calB_{p,q}^m$ class can be also defined
by 
\be
\calB_{p,q}^m=\{T\in \calL(\mv, L^2): \Big\|\|T\pi(z)^*g_0\|_{L^2}\Big\|_{\lpqm(\rdd)}< \fy\}.
\ee
We write $\calB_{p}^m=\calB_{p,p}^m$. If $m\equiv 1$, denote $\calB_{p,q}=\calB_{p,q}^m$.

At first view, this definition with $p=q=1$ and $m=v$ coincides with the corresponding definition in Theorem \ref{thm-M},
except the description of $T\in \calL(\mv, L^2)$. 
We point out that these two descriptions lead to the same operator class.
This fact will be clarified in Proposition \ref{pp-ebsb1} and Remark \ref{rk-S1v}.

Observe that the modulation space $M^{p,q}_{\tilde{m}}$ can be naturally isometric embedded into $\bpqm$ by
\be
f \longmapsto g_0\otimes f,
\ee
with 
\be
\|g_0\otimes f\|_{\bpqm}
= \Big\|\|g_0 \lan f, \pi(z)^*g_0\ran_{\mvs,\mv} \|_{L^2}\Big\|_{\lpqm(\rdd)}
= \Big\|V_{g_0}f(-z)\Big\|_{\lpqm(\rdd)}
= \|f\|_{M^{p,q}_{\tilde{m}}}.
\ee
In this sense, modulation space $\mpqm$ can be regarded as a closed subspace of $\bpqm$.

Like the case of modulation spaces, the $\calB_{p,q}^m$ classes also have some similar embedding relations.

\begin{proposition}\label{pp-ebsp}
	Let $p_i,q_i\in [1,\fy]$, $m_i\in \calM_v$, $i=1,2$. If $p_2\geq p_1$, $q_2\geq q_1$ and $m_2\leq m_1$,
	we have
	\be
	\calB_{p_1,q_1}^{m_1}\subset \calB_{p_2,q_2}^{m_2}\ \ \ \text{with}\ \ \ \ 	\|T\|_{\calB_{p_2,q_2}^{m_2}}\lesssim \|T\|_{\calB_{p_1,q_1}^{m_1}}
	\ee
	for all $T\in \calB_{p_1,q_1}^{m_1}$.
\end{proposition}
\begin{proof}
	Without loss of generality, we assume $m_1=m_2=m$.
	Let $\{e_n\}_n$ be a orthonormal basis of $L^2(\rd)$.
	For $T\in \calL(\mv,L^2)$, we have
	\ben\label{pp-ebsp-2}
	\begin{split}
	\|T\pi(z)^*g_0\|_{L^2}
	= &
	\left(\sum_{n=1}^{\fy}|\lan T\pi(z)^*g_0, e_n\ran_{L^2}|^2\right)^{1/2}
	\\
	= &
	\left(\sum_{n=1}^{\fy}|\lan T^*e_n, \pi(z)^*g_0\ran_{\mvs,\mv}|^2\right)^{1/2}
	=
	\left(\sum_{n=1}^{\fy}\big|V_{g_0}(T^*e_n)(-z)\big|^2\right)^{1/2}.
	\end{split}
	\een
	By Lemma \ref{lm-cSTFT}, we have
	\be
	|V_{g_0}(T^*e_n)|\leq |V_{g_0}g_0|\ast |V_{g_0}(T^*e_n)|.
	\ee
	Using this and a randomization method as in the proof of Propositions \ref{pp-cwi} and \ref{pp-cwid}, we find that
	\ben\label{pp-ebsp-1}
	\|T\pi(z)^*g_0\|_{L^2}\lesssim \|T\pi(z)^*g_0\|_{L^2}\ast |V_{g_0}g_0|.
	\een
	By the mixed-norm Young's inequality $L^{p_1,q_1}_m\ast L^{p_3,q_3}_v\subset L^{p_2,q_2}_m$
	with
	\be
	1+\frac{1}{p_2}=\frac{1}{p_1}+\frac{1}{p_3},\ \ 1+\frac{1}{q_2}=\frac{1}{q_1}+\frac{1}{q_3},
	\ee
	we obtain
	\be
	\Big\|\|T\pi(z)^*g_0\|_{L^2}\Big\|_{L^{p_2,q_2}_m}
	\lesssim
	\Big\|\|T\pi(z)^*g_0\|_{L^2}\Big\|_{L^{p_1,q_1}_m} 
	\Big\|V_{g_0}g_0\Big\|_{L^{p_3,q_3}_v}.
	\ee
	This proof has been finished.
\end{proof}
Recall that $v^{-1}(z)\lesssim m(z)\lesssim v(z)$, we have an immediate conclusion as follows. 
\begin{corollary}
	Let $p,q\in [1,\fy]$ and $m\in \calM_v$. We have the continuous embedding relation
	 \be
	\bv \subset \bpqm \subset \bvv=\calL(\mv,L^2).
	\ee
\end{corollary}

\begin{proposition}
	Let $p,q\in [1,\fy]$ and $m\in \calM_v$, the $\calB_{p,q}^m$ class is a Banach space.
\end{proposition}
\begin{proof}
For a Cauchy sequence $\{T_n\}_n$ in $\calB_{p,q}^m$, by the continuous embedding $\bpqm \subset \calL(\mv,L^2)$, 
$\{T_n\}_n$ is also Cauchy in $\calL(\mv,L^2)$.
Since $\calL(\mv,L^2)$ is a Banach space,  there exists an operator $T\in \calL(\mv,L^2)$ such that $T_n$ tends to $T$ in the topology of operator norm in  $\calL(\mv,L^2)$. 
For every $z\in \zd$, we have
\be
\lim_{n\rightarrow \fy}\|T_n\pi(z)^*g_0\|_{L^2}=\|T\pi(z)^*g_0\|_{L^2}.
\ee
By using Fatou's lemma, we verify $T\in \bpqm$ by
\be
\begin{split}
   \Big\|\|T\pi(z)^*g_0\|_{L^2}\Big\|_{\lpqm}
   = &
   \Big\|\liminf_{n\rightarrow \fy}\|T_n\pi(z)^*g_0\|_{L^2}\Big\|_{\lpqm}
   \\
   \leq & 
   \liminf_{n\rightarrow \fy}\Big\|\|T_n\pi(z)^*g_0\|_{L^2}\Big\|_{\lpqm}=\liminf_{n\rightarrow \fy}\|T_n\|_{\bpqm}\leq C.
\end{split}
\ee
For sufficiently large $n, m$, we have
\be
\|T_m-T_n\|_{\calB_{p,q}^m}=
\Big\|\|(T_m-T_n)\pi(z)^*g_0\|_{L^2}\Big\|_{\lpqm}<\ep.
\ee
Letting $m\rightarrow \fy$, we obtain for sufficiently large $n$
\be
\begin{split}
\Big\|\|(T-T_n)\pi(z)^*g_0\|_{L^2}\Big\|_{\lpqm}
= &
\Big\|\liminf_{m\rightarrow \fy}\|(T_m-T_n)\pi(z)^*g_0\|_{L^2}\Big\|_{\lpqm}
\\
\leq & 
\liminf_{m\rightarrow \fy}\Big\|\|(T_m-T_n)\pi(z)^*g_0\|_{L^2}\Big\|_{\lpqm}<\ep.
\end{split}
\ee
From this, we have that  $T_n$ tends to $T$ in the topology of $\bpqm$. This proof is completed.
\end{proof}

\subsection{Connection with the Schatten class}
\begin{proposition}\label{pp-ebsb1}
	Let $1\leq p\leq 2$.
	The following embedding relation is valid,
	\be
	\calB_p \subset \calS_p
	\ee
	in the sense that every $T\in \calB_p$ can be extended to a bounded operator on $\calL(L^2)$
	with
	\be
	\|T\|_{\calS_p}\leq \|T\|_{\calB_p}.
	\ee
	In particular, for $p=2$ we have
	\be
	\calB_2 = \calS_2.
	\ee
\end{proposition}
\begin{proof}
	First, we verify $\calB_2 = \calS_2$. This fact has been proved in Theorem \ref{thm-L2} with $T\in \calL(L^2)$.
	The only remaining thing we have to do is checking the proof for $T\in \calL(\mv, L^2)$.
	As in the proof of Theorem \ref{thm-L2}, we find that for $T\in \calL(\mv, L^2)$,
	\ben\label{pp-semb-1}
	\begin{split}
		\|T\|_{\calB_2}
		=\left(\sum_{n=1}^{\fy}\|T^*e_n\|_{L^2}^2\right)^{1/2}.
	\end{split}
    \een
	From this, we conclude that $T^*\in \calS_2$. 
	Using this and the fact that $\mv$ is dense in $L^2$, 
	we conclude that the operator $T\in \calB_2$ can be uniquely extended to a bounded operator on $L^2(\rd)$ satisfying 
	\be
	\|T\|_{\calS_2}=\|T^*\|_{\calS_2}=\|T\|_{\calB_2}.
	\ee
	This completes the proof of $\calB_2\subset \calS_2$. The inverse direction follows directly by \eqref{pp-semb-1}.

	Now, we turn to the proof of $\calB_p \subset \calS_p$.
	Using Proposition \ref{pp-ebsp}, we conclude that
	\be
	\calB_p\subset \calB_2=\calS_2,\ \ \ \ 1\leq p\leq 2.
	\ee
	Then, the operator $T\in \calB_p$ is compact, so it can be decomposed by
	\be
	T=\sum_{j=1}^{\fy}\la_j \xi_j\otimes \eta_j,
	\ee
	where $(\la_j)$ denotes the singular values of $T$, $(\xi_j)_j$ and $(\eta_j)_j$ are two orthonormal systems on $L^2(\rd)$.
	With this decomposition, we have
	\be
	\begin{split}
		\|T\pi(z)^*g_0\|_{L^2}
		= 
		\Big(\sum_{j=1}^{\fy}\la_j^2\big|\lan \pi(z)^*g_0, \eta_j \ran\big|^2\Big)^{1/2}
		\geq 
		\Big(\sum_{j=1}^{\fy}\la_j^p\big|\lan \pi(z)^*g_0, \eta_j \ran\big|^2\Big)^{1/p},
	\end{split}
	\ee
	where we use the H\"{o}lder inequality with the fact
	\be
	\sum_{j=1}^{\fy}\big|\lan \pi(z)^*g_0, \eta_j \ran\big|^2\leq \|\pi(z)^*g_0\|_{L^2}^2=\|g_0\|_{L^2}^2=1.
	\ee
	Then
	\be
	\begin{split}
		\|T\|_{\calB_p}
		= &
		\Big(\int_{\rdd}\|T\pi(z)^*g_0\|^p_{L^2}dz\Big)^{1/p}
		\\
		\geq &
		\Big(\int_{\rdd}\sum_{j=1}^{\fy}\la_j^p\big|\lan \pi(z)^*g_0, \eta_j \ran\big|^2dz\Big)^{1/p}
		\\
		= &
		\Big(\sum_{j=1}^{\fy}\la_j^p\int_{\rdd}\big|\lan \pi(z)^*g_0, \eta_j \ran\big|^2dz\Big)^{1/p}
		=
		\Big(\sum_{j=1}^{\fy}\la_j^p\|V_{g_0} \eta_j\|_{L^2(\rdd)}^2\Big)^{1/p}=\Big(\sum_{j=1}^{\fy}\la_j^p\Big)^{1/p}.
	\end{split}
	\ee
	We have $\calB_p \subset \calS_p$ with $\|T\|_{\calS_p}\leq \|T\|_{\calB_p}$. 
\end{proof}

\begin{remark}\label{rk-S1v}
	Let $v$ be a submultiplicative weight. Recall that $v(z)\geq 1$.
	Using this proposition, we have $\calB_1^v\subset \calB_1\subset \calS_1$.
	Then the operator classes $\calB_1^v$ can be re-represented as
	$$\calB_1^v:=\{S\in \calS_1: \Big\|\|S\pi(z)^*g_0\|_{L^2}\Big\|_{L^1_v(\rdd)} <\infty\}$$
	or as the definition in Theorem \ref{thm-M}, that is, $\calB_1^v:=\{S\in \calL(L^2): \Big\|\|S\pi(z)^*g_0\|_{L^2}\Big\|_{L^1_v(\rdd)} <\infty\}$.
\end{remark}

\begin{proposition}\label{pp-ebsb2}
	Let $2< p\leq \fy$.
	We have the following embedding relation
	\be
	\calS_p \subset \calB_p
	\ee
	with
	\be
	\|T\|_{\calB_p}\leq \|T\|_{\calS_p}.
	\ee
\end{proposition}
\begin{proof}
	The case of $p=\fy$ follows by
	\be
	\calS_{\fy}=\calL(L^2)\subset \calL(M^1,L^2)=\calB_{\fy}.
	\ee
	For $T\in \calS_p$ with $p\in (2,\fy)$ we write 
	\be
	T=\sum_{j=1}^{\fy}\la_j \xi_j\otimes \eta_j
	\ee
	with singular values $(\la_j)$, where $(\xi_j)_j$ and $(\eta_j)_j$ are two orthonormal systems on $L^2(\rd)$.
	
	Using a similar way as in the proof of Proposition \ref{pp-ebsb1}, we conclude that
	\be
	\begin{split}
		\|T\pi(z)^*g_0\|_{L^2}
		= 
		\Big(\sum_{j=1}^{\fy}\la_j^2\big|\lan \pi(z)^*g_0, \eta_j \ran\big|^2\Big)^{1/2}
		\leq 
		\Big(\sum_{j=1}^{\fy}\la_j^p\big|\lan \pi(z)^*g_0, \eta_j \ran\big|^2\Big)^{1/p}.
	\end{split}
	\ee
	Then
	\be
	\begin{split}
		\|T\|_{\calB_p}
		= 
		\Big(\int_{\rdd}\|T\pi(z)^*g_0\|^p_{L^2}dz\Big)^{1/p}
		\leq
		\Big(\sum_{j=1}^{\fy}\la_j^p\|V_{g_0} \eta_j\|_{L^2(\rdd)}^2\Big)^{1/p}=\Big(\sum_{j=1}^{\fy}\la_j^p\Big)^{1/p}.
	\end{split}
	\ee
	We have $\calS_p \subset \calB_p$ with $\|T\|_{\calB_p}\leq \|T\|_{\calS_p}$. 
\end{proof}
\begin{remark}
	For $p\in [1,2)$, take $T_1=f\otimes g$ with $f\in L^2, g \in L^2\bs M^p$. One can verify that $T_1\in \calS_1\bs \calB_p\subset \calS_p\bs \calB_p$.
	From this and the relation $\calB_p\subset \calS_p$, we conclude that $\calB_p$ is a proper subset of $\calS_p$ for $p\in [1,2)$.
	If $p\in (2,\fy]$,  take $T_2=f\otimes g$ with $f\in L^2, g \in M^p\bs L^2$. We find that $T_2\in \calB_p\bs \calS_{\fy}\subset \calB_p\bs \calS_p$.
	This and the relation $\calS_p\subset \calB_p$ implies that $\calS_p$ is a proper subset of $\calB_p$ for $p\in (2,\fy]$.
	We also point out that for $p\in (2,\fy]$, $\calB_p$ is no longer a subset of $\calL(L^2)$, which is quite different from  the Schatten class.
\end{remark}

\subsection{$\bpqm$ as the operator-valued modulation spaces} 
According to the previous description in this paper, the window class $\calB_1^v$ is an extension of the classical window class $\mv$,
and the operator class $\calB_{p,q}^m$ is a generalization of $\calB_1^v$.
In Subsection 4.2, we point out that the classical modulation space $M^{p,q}_m$ can be isometric embedded into $\bpqm$.
Here, from another perspective, we will make clear that the operator classes $\calB_{p,q}^m$ can be exactly regarded as the modulation spaces in the level of operators on $\calL(\mv,L^2)$.

For any $f\in \mvs=\calL(\mv, \bbC)$, recall the classical STFT by
\be
V_{g_0}f(z)=\lan f, \pi(z)g_0\ran_{\mvs,\mv}.
\ee
We extend the STFT by
\be
\scrV_{g_0}T(z)=T(\pi(z)g_0),\ \ \ \ T\in \calL(\mv,L^2).
\ee
Then the $\calB_{p,q}^m$ norm can be re-represented by
\be
\|T\|_{\calB_{p,q}^m}=\Big\|\|T\pi(z)^*g_0\|_{L^2}\Big\|_{L^{p,q}_m}=\Big\|\scrV_{g_0}T(z) \Big\|_{L^{p,q}_{\tilde{m}}(\rdd, L^2)}.
\ee
Define the operator-valued modulation spaces $\scrM^{p,q}_m$ by
\be
\scrM^{p,q}_m=\{T\in \calL(\mv,L^2): \Big\|\scrV_{g_0}T(z) \Big\|_{L^{p,q}_{m}(\rdd,L^2)}<\fy\}.
\ee
We have the equivalent relation
\be
\calB_{p,q}^m= \scrM^{p,q}_{\tilde{m}}.
\ee
From this, the basic properties of $\bpqm$ established above can be naturally transferred to $\scrM^{p,q}_{\tilde{m}}$.
Specifically, we have
\be
\scrM^1_v \subset \scrM^{p,q}_m \subset \scrM^{\fy}_{v^{-1}}=\calL(\mv,L^2),\ \ \ 1\leq p,q\leq \fy.
\ee
Note that the distribution space $\mvs=\calL(\mv, \bbC)$ in classical modulation spaces
is replaced by the ``operator-valued distribution space''  $\calL(\mv,L^2)$.

Now that we know $\bpqm$ can be regarded as the modulation spaces in the level of operators on $\calL(\mv,L^2)$.
Naturally, all the properties of classical modulation spaces are expected to be represented in the corresponding
operator-valued modulation spaces or operator classes.
Here, we only point out the independence of window functions.

\begin{proposition}[The window functions of $\bpqm$]\label{pp-winsp}
	The definition of $\bpqm$ is independent of the window $\va\in \mv\bs\{0\}$.
	Different windows yield equivalent norms. 
	For all $\va\in \mv\bs\{0\}$, we have
\be
\calB_{p,q}^m=\{T\in \calL(\mv, L^2): \Big\|\|T\pi(z)^*\va\|_{L^2}\Big\|_{\lpqm(\rdd)}< \fy\}.
\ee
\end{proposition}
\begin{proof}
For $T\in \calL(\mv,L^2)$, $\va\in \mv$, 
by a similar calculation as in the proof of Proposition \ref{pp-ebsp}, we have
\be
\begin{split}
	\|T\pi(z)^*g_0\|_{L^2}
	=
	\left(\sum_{n=1}^{\fy}\big|V_{g_0}(T^*e_n)(-z)\big|^2\right)^{1/2},\ \ \ 
	\|T\pi(z)^*\va\|_{L^2}
	=
	\left(\sum_{n=1}^{\fy}\big|V_{\va}(T^*e_n)(-z)\big|^2\right)^{1/2}.
\end{split}
\ee

Using Lemma \ref{lm-cSTFT} and a randomization method, we conclude that
\be
\|T\pi(z)^*\va\|_{L^2}\lesssim \|T\pi(z)^*g_0\|_{L^2}\ast |V_{g_0}\va|
\ee
and
\be
\|T\pi(z)^*g_0\|_{L^2}\lesssim \|\va\|_{L^2}^{-2}\|T\pi(z)^*\va\|_{L^2}\ast |V_{\va}g_0|.
\ee
By the Young inequality $\lpqm\ast L^1_v\subset \lpqm$, we conclude that
\be
\Big\|\|T\pi(z)^*\va\|_{L^2}\Big\|_{\lpqm} 
\lesssim \Big\|\|T\pi(z)^*g_0\|_{L^2}\Big\|_{\lpqm} \|V_{g_0}\va\|_{L^1_v}
\ee
and
\be
\Big\|\|T\pi(z)^*g_0\|_{L^2}\Big\|_{\lpqm} 
\lesssim 
\|\va\|_{L^2}^{-2} \Big\|\|T\pi(z)^*\va\|_{L^2}\Big\|_{\lpqm} \|V_{\va}g_0\|_{L^1_v}.
\ee
We have now completed this proof.
\end{proof}

\subsection{Application to the localization operators}

The localization operator $\calA_a^{\va_1,\va_2}$ with symbol $a\in \calS(\rdd)$, analysis window $\va_1$ and synthesis window $\va_2$ is defined formally by means of the STFT as
\be
\calA_a^{\va_1,\va_2}f=V_{\va_2}^*(aV_{\va_1}f)
=
\int_{\rdd} a(x,\xi)V_{\va_1}f(x,\xi)M_{\xi}T_x\va_2dxd\xi
\ee
whenever the  vector-valued integral makes sense. 
Usually, it is more convenient to interpret the definition of localization operator in a weak sense as follows
\be
\lan \calA_a^{\va_1,\va_2}f,\ g\ran_{\calS'(\rd),\calS(\rd)}=\lan a,\overline{ V_{\va_1}f} V_{\va_2}g \ran_{\calS'(\rdd),\calS(\rdd)},\ \ \ \ f,g\in \calS(\rd),
\ee
where the right term makes sense for $a\in \calS'(\rdd)$ and $\va_1,\va_2,f,g\in \calS(\rd)$. From this, one can find that $\calA_a^{\va_1,\va_2}$
is a well-defined continuous operator from $\calS(\rd)$ into $\calS'(\rd)$.

Here, we focus on the property of symbol $a$ when the corresponding localization operator $\calA_a^{\va_1,\va_2}$ belongs to the operator classes defined in this section. 
First, we recall a classical theorem connecting the symbols and the corresponding localization operators of $\calS_p$ classes.
See also \cite{Cordero2003}.

\begin{lemma}\cite[Theorem 1]{Cordero2005}\label{lm-losp}
	Let $1\leq p\leq \fy$. 
	\bn
	\item
	The mapping $(a,\va_1,\va_2)\longmapsto \calA^{\va_1,\va_2}_a$ is bounded from $M^{p,\fy}(\rdd)\times M^1(\rd)\times M^1(\rd)$ into $\calS_p$ with a norm estimate
	\be
	\|\calA^{\va_1,\va_2}_a\|_{\calS_p}\leq B\|a\|_{M^{p,\fy}(\rdd)}\|\va_1\|_{M^1(\rd)}\|\va_2\|_{M^1(\rd)}.
	\ee
	\item
	Conversely, if $\calA^{\va_1,\va_2}_a\in \calS_p$ for all windows $\va_1,\va_2\in \mv(\rd)$
	with 
	\be
	\|\calA^{\va_1,\va_2}_a\|_{\calS_p}
	\leq 
	B\|\va_1\|_{M^1(\rd)}\|\va_2\|_{M^1(\rd)},
	\ee
	then $a\in M^{p,\fy}(\rdd)$.
	\en
\end{lemma}

For the connection of the symbols and the corresponding localization operators of $\calB_p$ classes, we give the following proposition.

\begin{theorem}\label{thm-loap}
		Let $1\leq p\leq \fy$, $a\in \calS'(\rdd)$.
	\bn
	\item
	The mapping $(a,\va_1,\va_2)\longmapsto \calA^{\va_1,\va_2}_a$ is bounded from $M^{p,\fy}(\rdd)\times M^1(\rd)\times M^1(\rd)$ into $\calB_p$ with a norm estimate
	\be
	\|\calA^{\va_1,\va_2}_a\|_{\calB_p}\leq B\|a\|_{M^{p,\fy}}\|\va_1\|_{M^1(\rd)}\|\va_2\|_{M^1(\rd)}.
	\ee
	\item
	Conversely, if $\calA^{\va_1,\va_2}_a\in \calB_p$ for all windows $\va_1,\va_2\in \mv(\rd)$
	with 
	\be
	\|\calA^{\va_1,\va_2}_a\|_{\calB_p}
	\leq 
	B_a\|\va_1\|_{M^1(\rd)}\|\va_2\|_{M^1(\rd)},
	\ee
	then $a\in M^{p,\fy}(\rdd)$ with $\|a\|_{M^{p,\fy}}\lesssim B_a$.
	\en
\end{theorem}
\begin{proof}
	We first verify the statement (1).
	Using Lemma \ref{lm-losp} and Proposition \ref{pp-ebsb2}, we conclude that for $p\in [2,\fy]$,
	\be
	\|\calA^{\va_1,\va_2}_a\|_{\calB_p}
	\lesssim 
	\|\calA^{\va_1,\va_2}_a\|_{\calS_p}
	\leq
	C\|a\|_{M^{p,\fy}(\rdd)}\|\va_1\|_{M^1(\rd)}\|\va_2\|_{M^1(\rd)}.
	\ee
	
	For the case $p\in [1,2)$, since $\calS_p\subset \calB_p$ is not valid, we need to deal with this case directly.
	For $a\in M^{p,\fy}(\rdd)$, $\va_1,\va_2\in  M^1(\rd)$, we find $\calA^{\va_1,\va_2}_a\in S_{\fy}=\calL(L^2)$ by using Lemma \ref{lm-losp}.
	For $z=(z_1,z_2)\in \rdd$ and $z'=(z'_1,z'_2)\in \rdd$, we write
	\be
	\lan \calA^{\va_1,\va_2}_a\pi(z)g_0, \pi(z')g_0\ran_{L^2}=\lan a, \overline{V_{\va_1}\pi(z)g_0} V_{\va_2}\pi(z')g_0\ran_{M^{\fy},M^1}.
	\ee
	By Lemma \ref{lm-FTPSTFT}, we have
	\be
	\scrF\big(\overline{V_{\va_1}\pi(z)g_0} V_{\va_2}\pi(z')g_0\big)(x,y)=\big(\overline{V_{\va_1}\va_2}V_{\pi(z)g_0}\pi(z')g_0\big)(-y,x).
	\ee
	Let $b(x,y)=a(y,-x)$. We have $\|b\|_{M^{p,\fy}}= \|a\|_{M^{p,\fy}}$.
	Write
	\be
	\begin{split}
		&
		\lan a, \overline{V_{\va_1}\pi(z)g_0} V_{\va_2}\pi(z')g_0\ran_{M^{\fy},M^1}
		\\
		= &
		\lan \scrF a, \scrF\big(\overline{V_{\va_1}\pi(z)g_0} V_{\va_2}\pi(z')g_0\big)\ran_{M^{\fy},M^1}
		\\
		= &
		\lan \scrF b, \overline{V_{\va_1}\va_2}V_{\pi(z)g_0}\pi(z')g_0\ran_{M^{\fy},M^1}
		=
		\lan \scrF b \cdot V_{\va_1}\va_2, V_{\pi(z)g_0}\pi(z')g_0\ran_{M^{\fy},M^1}.
	\end{split}
    \ee
    By a direct calculation, we conclude that
    \be
    V_{\pi(z)g_0}\pi(z')g_0=e^{2\pi iz_1'(z_2'-z_2)}M_{(z_2,-z_1')}T_{(z'-z)}V_{g_0}g_0.
    \ee
	The above two estimates imply that
	\be
	\begin{split}
	\big|\lan a, \overline{V_{\va_1}\pi(z)g_0} V_{\va_2}\pi(z')g_0\ran_{M^{\fy},M^1}\big|
	= &
	\big|\lan \scrF b \cdot V_{\va_1}\va_2, M_{(z_2,-z_1')}T_{(z'-z)}V_{g_0}g_0 \ran_{M^{\fy},M^1}\big|
	\\
	= &
	\big|\lan \scrF b \cdot V_{\va_1}\va_2, M_{(z_2,-z_1')}T_{(z'-z)}\Phi \ran_{M^{\fy},M^1}\big|
	\\
	= &
	|V_{\Phi}(\scrF b \cdot V_{\va_1}\va_2)(z'-z,(z_2,-z_1'))|.
    \end{split}
	\ee
	Here, we denote $\Phi=|V_{g_0}g_0|$. Combining the convolution relation $M^{p,\fy}\ast M^1\subset M^p$ (see Lemma \ref{lm-covm}) with
	\be
	\|V_{\va_1}\va_2\|_{M^1(\rdd)}\lesssim \|\va_1\|_{M^1(\rd)}\|\va_2\|_{M^1(\rd)},
	\ee
	we conclude that
	\be
	\begin{split}
	\|\scrF b \cdot V_{\va_1}\va_2\|_{M^p(\rdd)}
	= &
	\|b\ast \scrF^{-1}(V_{\va_1}\va_2)\|_{M^p(\rdd)}
	\\
	\lesssim &
	\|b\|_{M^{p,\fy}(\rdd)}\|\scrF^{-1}(V_{\va_1}\va_2)\|_{M^1(\rdd)}
	\\
	= &
	\|b\|_{M^{p,\fy}(\rdd)}\|V_{\va_1}\va_2\|_{M^1(\rdd)}
	\lesssim 
	\|b\|_{M^{p,\fy}(\rdd)}\|\va_1\|_{M^1(\rd)}\|\va_2\|_{M^1(\rd)}.
    \end{split}
	\ee
	Denote by $F=\scrF b \cdot V_{\va_1}\va_2\in M^p(\rdd)$. Let
	\be
	\scrA: (z_1', z_2', z_1, z_2) \longmapsto (z'-z, z_2, -z_1')
	\ee
	be a inverse linear transform from $\rddd$ into $\rddd$.
	We write
	\be
	|V_{\Phi}(\scrF b \cdot V_{\va_1}\va_2)(z'-z,(z_2,-z_1'))|=|\big(V_{\Phi}F\circ \scrA\big) (z_1', z_2', z_1, z_2)|.
	\ee
	Form the following inequality, 
	\be
	|V_{\Phi}F|\lesssim |V_{\Phi}\Phi|\ast|V_{\Phi}F|,
	\ee
	we obtain
	\be
	|V_{\Phi}F\circ \scrA|\lesssim |V_{\Phi}\Phi \circ \scrA|\ast|V_{\Phi}F \circ \scrA|.
	\ee
	Using Young's inequality $L^{r,1}(\rdd\times \rdd)\ast L^{p,p}(\rdd\times \rdd)\subset L^{2,p}(\rdd\times \rdd)$
	with
	$
	1+1/2=1/r+1/p,
	$
	we conclude that
	\be
	\begin{split}
	\|V_{\Phi}F\circ \scrA\|_{L^{2,p}(\rdd\times \rdd)}
	\lesssim &
	\|V_{\Phi}\Phi \circ \scrA\|_{L^{r,1}(\rdd\times \rdd)}
	\|V_{\Phi}F \circ \scrA\|_{L^{p,p}(\rdd\times \rdd)}
	\\
	\lesssim &
	\|V_{\Phi}F \circ \scrA\|_{L^{p,p}(\rdd\times \rdd)}
	\sim 
	\|V_{\Phi}F\|_{L^{p,p}(\rdd\times \rdd)}=\|F\|_{M^p(\rdd)}.
	\end{split}
	\ee
	The desired conclusion follows by
	\be
	\begin{split}
		\|\calA^{\va_1,\va_2}_a\|_{\calB_p}
		= &
		\Big\|\|\calA^{\va_1,\va_2}_a\pi(z)g_0\|_{L^2(\rd)}\Big\|_{L^p(\rdd)}
		\\
		= &
		\Big\|\|\lan \calA^{\va_1,\va_2}_a\pi(z)g_0, \pi(z')g_0\ran \|_{L^2(\rdd)}\Big\|_{L^p(\rdd)}
		\\
		= &
		\Big\|\|V_{\Phi}(\scrF b \cdot V_{\va_1}\va_2)(z'-z,(z_2,-z_1')) \|_{L^2(\rdd)}\Big\|_{L^p(\rdd)}
		\\
		= &
		\|V_{\Phi}F\circ \scrA\|_{L^{2,p}(\rdd\times \rdd)}
		\\
		\lesssim &
		\|F\|_{M^p(\rdd)}=\|\scrF b \cdot V_{\va_1}\va_2\|_{M^p(\rdd)}
		\\
		\lesssim &
		\|b\|_{M^{p,\fy}(\rdd)}\|\va_1\|_{M^1(\rd)}\|\va_2\|_{M^1(\rd)}
		\\
		=&
		\|a\|_{M^{p,\fy}(\rdd)}\|\va_1\|_{M^1(\rd)}\|\va_2\|_{M^1(\rd)}.
	\end{split}
	\ee
	
	Next, we turn to the proof of statement (2). 
	Take $\Psi=|V_{g_0}g_0|^2$.
	A direct calculation  (see also \cite[Lemma 1]{Cordero2005}) yields that
	\be
	\begin{split}
	&|V_{\Psi}a(z,\z)|=|\lan a, M_{\z}T_z(\overline{V_{g_0}g_0}V_{g_0}g_0)\ran|
	\\
	= &
	|\lan a, \overline{V_{g_0}(M_{z_2}T_{z_1}g_0)}V_{M_{\z_1}T_{-\z_2}g_0}(M_{\z_1}T_{-\z_2}M_{z_2}T_{z_1}g_0)\ran|
	\\
	= &
	|\lan A_a^{g_0,M_{\z_1}T_{-\z_2}g_0}M_{z_2}T_{z_1}g_0,\ M_{\z_1}T_{-\z_2}M_{z_2}T_{z_1}g_0\ran|
	\leq 
	\|A_a^{g_0,M_{\z_1}T_{-\z_2}g_0}M_{z_2}T_{z_1}g_0\|_{L^2}.
	\end{split}
	\ee
	For any $\z\in \rdd$, we conclude that
	\be
\begin{split}
	\|V_{\Psi}a(\cdot,\z)\|_{L^p}
	\leq &
	\Big\|\|A_a^{g_0,M_{\z_1}T_{-\z_2}g_0}\pi(z)g_0\|_{L^2}\Big\|_{L^p}
	\\
	= &
	\Big\|\|A_a^{g_0,M_{\z_1}T_{-\z_2}g_0}\pi(z)^*g_0\|_{L^2}\Big\|_{L^p}
	\\
	= &
	\|A_a^{g_0,M_{\z_1}T_{-\z_2}g_0}\|_{\calB_p}
	\leq 
	B_a\|g_0\|_{M^1}\|M_{\z_1}T_{-\z_2}g_0\|_{M^1}
	=
	B_a\|g_0\|_{M^1}^2.
\end{split}
\ee
We have now completed the proof of statement (2).
\end{proof}
\begin{remark}
In contrast to Lemma \ref{lm-losp},
the conclusion in Theorem \ref{thm-loap}  is somewhat surprising, since in general we have  $\calB_p\subsetneq \calS_p$ for $p<2$, and $\calS_p\subsetneq \calB_p$ for $p>2$.
A reason for this surprising conclusion may be from the overly nice property of the window functions. 
\end{remark}

Next, we turn our attention to the Cohen class $Q_{T}$ for $T=\calA^{\va,\va}_a$.
In this case, we assume that $a\geq 0$. Then $T$ is a positive operator.

	Denote by $Q_0=[-1/2,1/2]^{2d}$ the unit cube of $\rdd$ centered at the origin.
	Let $p,q\in (0,\fy]$.
	We recall that the Wiener amalgam space $W(L^p,L^q_v)(\rdd)$ consists of all measurable functions for which
	the following norm are finite:
	\be
	\|f\|_{W(L^p,L^q_v)(\rdd)}:= \bigg(\sum_{k,n\in \zd}\|f T_{(k,n)}\chi_{Q_0}\|^q_{L^p(\rdd)}v(k,n)^q\bigg)^{1/q},
	\ee
	with the usual modification when $q=\fy$.

\begin{lemma}\label{lm-Wiener}
	Let $\Phi(z)=|V_{g_0}g_0(z)|^2=e^{-\pi|z|^2}$ for $z\in \zdd$.  We have the following equivalent norm of $W(L^1, L^{1/2}_v)$
	\be
	\|a\|_{W(L^1, L^{1/2}_v)}\sim \big\|\|aT_{z}\Phi\|_{L^1}\big\|_{L^{1/2}_v(\rdd)}.
	\ee
	If $v$ grows at most polynomial, the window function $\Phi$ can be replaced by any nonzero Schwartz function.
\end{lemma}
\begin{proof}
	For $z\in Q_0+(k,n)$, we have $T_{(k,n)}\chi_{Q_0}\lesssim T_z\Phi$. Then
	\be
	\|a T_{(k,n)}\chi_{Q_0}\|_{L^p(\rdd)}\lesssim \|a T_z\Phi \|_{L^p(\rdd)},\ \ \ z\in Q_0+(k,n).
	\ee
	From this and the fact $v(z)\sim v(k,n)$ for $z\in Q_0+(k,n)$. We conclude that
	\be
	\bigg(\sum_{k,n\in \zd}\|a T_{k,n}\chi_{Q_0}\|^{1/2}_{L^1(\rdd)}v(k,n)^{1/2}\bigg)^{2}\lesssim \big\|\|aT_{z}\Phi\|_{L^1}\big\|_{L^{1/2}_v(\rdd)}.
	\ee
	This estimate is also valid when $\Phi$ is replaced by any nonzero Schwartz function, since that any non-zero continuous function has a positive lower bound on a sufficiently small cube.
	
	Let $\Phi_1(z)=e^{\frac{-\pi|z|^2}{2}}$ and $\Phi_2(z)=e^{\frac{-\pi|z|^2}{4}}$.
	For the inverse direction, notice that for $z\in Q_0+(k,n)$ we have
	\be
	T_{z}\Phi\lesssim T_{(k,n)}\Phi_1\ \ \ \text{and} \ \ \ \|a T_{z}\Phi\|_{L^1(\rdd)}\lesssim \|a T_{(k,n)}\Phi_1\|_{L^1(\rdd)}.
	\ee
	Using a similar method as above, we find that
	\be
	\big\|\|aT_{z}\Phi\|_{L^1}\big\|_{L^{1/2}_v(\rdd)}\lesssim 
	\bigg(\sum_{k,n\in \zd}\|a T_{(k,n)}\Phi_1\|^{1/2}_{L^1(\rdd)}v(k,n)^{1/2}\bigg)^{2}.
	\ee
	
	Next, we write
	\be
	\begin{split}
	\|aT_{(k,n)}\Phi_1\|_{L^1}
	\leq &
	\sum_{j,l\in \zd}\|aT_{(k,n)}\Phi_1\cdot  T_{(j,l)}\chi_{Q_0}\|_{L^1}
	\\
	\leq &
	\sum_{j,l\in \zd}\|a T_{(j,l)}\chi_{Q_0}\|_{L^1}\|T_{(k,n)}\Phi_1\cdot T_{(j,l)}\chi_{Q_0}\|_{L^{\fy}}
	\\
	\lesssim &
	\sum_{j,l\in \zd}\|a T_{(j,l)}\chi_{Q_0}\|_{L^1}\Phi_2((k,n)-(j,l)).
	\end{split}
	\ee
	Form this and a convolution inequality $l^{1/2}_v\ast l^{1/2}_v\subset l^{1/2}_v$, we conclude that
	\be
	\begin{split}
	&\bigg(\sum_{k,n\in \zd}\|a T_{(k,n)}\Phi_1\|^{1/2}_{L^1(\rdd)}v(k,n)^{1/2}\bigg)^{2}
	\\
	\lesssim &
	\bigg(\sum_{k,n\in \zd}\|a T_{(k,n)}\chi_{Q_0}\|^{1/2}_{L^1(\rdd)}v(k,n)^{1/2}\bigg)^{2}\|\Phi_2\|_{l^{1/2}_v}\lesssim  \bigg(\sum_{k,n\in \zd}\|a T_{(k,n)}\chi_{Q_0}\|^{1/2}_{L^1(\rdd)}v(k,n)^{1/2}\bigg)^{2}.
	\end{split}
	\ee
	For the case of submultiplicative weight $v$ with at most polynomial growth, we notice that
	\be
	\Phi(z)\leq C_{\Phi,N}(1+|z|)^{-N},\ \ \ \ z\in \zdd, N\geq 1,
	\ee
	for any Schwartz function $\Phi$.
	Then the above argument still works in this case. We have now completed the whole proof.
\end{proof}

\begin{proposition}\label{pp-CohenLocal}
	Let $1\leq p,q\leq \fy$. Suppose that $a\geq 0$ be a measurable function on $\rdd$.
	Let $v$ be a submultiplicative weight.	
	For any $\va\in \calS(\rd)\bs\{0\}$ and $\Psi=|V_{\va}g_0|^2$, we have
	\be
\sqrt{\calA^{\va,\va}_a}\in \calB_1^v\ \text{and}\  a\in M^{\fy}
\Longleftrightarrow 
\big\|\|aT_{z}\Psi\|_{L^1}\big\|_{L^{1/2}_{v^2}(\rdd)}<\fy
\ee
	with
	\be
	\|\sqrt{\calA^{\va,\va}_a}\|_{\calB_1^v}\sim \sqrt{\big\|\|aT_{z}\Psi\|_{L^1}\big\|_{L^{1/2}_{v^2}(\rdd)}}.
	\ee
	Furthermore, if $\va=g_0$ or $v$ has at most polynomial growth, we have
	\be
	\sqrt{\calA^{\va,\va}_a}\in \calB_1^v\ \text{and}\  a\in M^{\fy}
	\Longleftrightarrow 
	a\in W(L^1, L^{1/2}_{v^2}).
	\ee
\end{proposition}
\begin{proof}
	Checking the proof of Lemma \ref{lm-Wiener}, we find that $\big\|\|aT_{z}\Psi\|_{L^1}\big\|_{L^{1/2}_{v^2}(\rdd)}<\fy$ implies 
	\be
	a\in W(L^1, L^{1/2}_{v^2})\subset W(L^1, L^{1/2})\subset W(L^1, L^{1})=L^1\subset M^{\fy}.
	\ee
	If $a\in M^{\fy}$ and $a\geq 0$, 
	by Lemma \ref{lm-losp} we have $\calA^{\va,\va}_a$ be a positive operator in $\calL(L^2)$. Write
	\be
	\begin{split}
	\|\sqrt{\calA^{\va,\va}_a}\pi(z)^*g_0\|_{L^2}^2
	= &
	\lan \calA^{\va,\va}_a\pi(z)^*g_0, \pi(z)^*g_0\ran_{L^2}
	\\
	= &
	\lan \pi(z)\calA^{\va,\va}_a\pi(z)^*g_0, g_0\ran_{L^2}
	\\
	= &
	\lan \calA^{\va,\va}_{T_za}g_0, g_0\ran_{L^2}
	=\lan T_za, \Psi\ran_{\calS',\calS}=\|aT_{-z}\Psi\|_{L^1}.
	\end{split}
    \ee
    The desired conclusion follows by
    \be
    \begin{split}
    	\|\sqrt{\calA^{\va,\va}_a}\|_{B_1^v}
    	=
    \|\|\sqrt{\calA^{\va,\va}_a}\pi(z)^*g_0\|_{L^2}\|_{L^1_v(\rdd)}
    = &
    \|\sqrt{\|aT_{-z}\Psi\|_{L^1}}\|_{L^1_v(\rdd)}
    = 
    \sqrt{\big\|\|aT_{z}\Psi\|_{L^1}\big\|_{L^{1/2}_{v^2}(\rdd)}}.
    \end{split}
    \ee
    
    If $\va=g_0$, notice that $\Psi(z)=|V_{g_0}g_0(z)|^2=e^{-\pi|z|^2}$, we obtain
    $\|a\|_{W(L^1, L^{1/2}_{v^2})}\sim \big\|\|aT_{z}\Psi\|_{L^1}\big\|_{L^{1/2}_{v^2}(\rdd)}$ from Lemma \ref{lm-Wiener}.
    If $v$ has at most polynomial growth, the same conclusion also follows by Lemma \ref{lm-Wiener}.
    Thus, the equivalent relation $\sqrt{\calA^{\va,\va}_a}\in \calB_1^v \Longleftrightarrow a\in W(L^1, L^{1/2}_{v^2})$ is valid.
\end{proof}

As a corollary, we have the following improvement of \cite[Proposition 8.4]{Skrettingland2022JoFAaA}.
\begin{corollary}
	Let $\va\in \calS(\rd)$. Let $a\in W(L^1, L^{1/2}_{v^2})$ be a non-negative function. Suppose that $v$ is a submultiplicative weight with at most polynomial growth and $m\in \calM_v$.
	We have
	\be
	\|f\|_{\mpqm(\rd)}\sim \|\sqrt{Q_{\calA^{\va,\va}_a}f(z)}\|_{\lpqm(\rdd)}.
	\ee
\end{corollary}
\begin{proof}
	By Proposition \ref{pp-CohenLocal}, we find that $\calA^{\va,\va}_a$ is a positive operator in $\calL(L^2)$, and
	 $\sqrt{\calA^{\va,\va}_a}\in \calB_1^v$. 
	Then the desired conclusion follows by Theorem \ref{thm-Mc}.
\end{proof}

\subsection*{Acknowledgements}
This work was supported by the Natural Science Foundation of Fujian Province
[2020J01708, 2020J01267,2021J011192].

\bibliographystyle{abbrv}

\end{document}